\documentclass[11pt,reqno,sumlimits]{amsart}

\usepackage[utf8]{inputenc}
\usepackage{amssymb, amscd, amsmath, epsfig, mathtools}
\usepackage{amsthm}
\usepackage{enumerate}
\usepackage{xcolor}
\usepackage{scalerel}
\usepackage{soul}
\usepackage{tikz-cd}
\usepackage{esint}

\usepackage[margin=1.0in]{geometry}

\newtheorem{theorem}{Theorem}[section]
\newtheorem*{theorem*}{Theorem}
\newtheorem{definition}{Definition}[section]
\newtheorem{corollary}{Corollary}[section]
\newtheorem{lemma}{Lemma}[section]
\newtheorem{proposition}{Proposition}[section]
\theoremstyle{definition}
\newtheorem{remark}{Remark}[section]

\newcommand{\R}{\mathbb R}

\newcommand{\calC}{\mathcal C}
\newcommand{\calL}{\mathcal L}
\newcommand{\dvol}{ d\text{Vol}_{g}}

\begin{document}

\title[Trichotomy Theorem: Prescribed Scalar and Mean Curvatures]{Trichotomy Theorem for Prescribed Scalar and Mean Curvatures on Compact Manifolds with Boundaries}
\author[J. Xu]{Jie Xu}
\address{
Department of Mathematics and Statistics, Boston University, Boston, MA, U.S.A.}
\email{xujie@bu.edu}
\address{
Institute for Theoretical Sciences, Westlake University, Hangzhou, Zhejiang Province, China}
\email{xujie67@westlake.edu.cn}

\date{}							% Activate to display a given date or no date

\maketitle

\begin{abstract} In this article, we give results of prescribing scalar and mean curvature functions for metrics either pointwise conformal or conformally equivalent to a Riemannian metric that is equipped on a compact manifold with boundary, with dimensions at least $ 3 $. The results are classified by the sign of the first eigenvalue of the conformal Laplacian. This leads to a ``Trichotomy Theorem" in terms of both scalar and mean curvature functions, which is a full extension of the ``Trichotomy Theorem" given by Kazdan and Warner. We also discuss prescribing Gauss and geodesic curvature problems on compact Riemann surfaces with boundary for metrics either pointwise conformal or conformally equivalent to the original metric, provided that the Euler characteristic is negative. The key step is a general version of monotone iteration scheme which handle the zeroth order nonlinear term on the boundary conditions.
\end{abstract}

% This is the beginning of the introductory section.
\section{Introduction}
In this article, we give a ``Trichotomy Theorem" on compact manifolds $ (\bar{M}, g) $ with non-empty smooth boundaries $ \partial M $, $ n : = \dim M \geqslant 3 $, involving both the scalar and mean curvatures. This is a full generalization of the ``Trichotomy Theorem" on closed manifolds, given by Kazdan and Warner \cite{KW3}. Precisely speaking, this ``Trichotomy Theorem" concerns whether the given functions $ S, H $ can be realized as scalar and mean curvatures, respectively, of a metric $ \tilde{g} $ either within a conformal class $ [g] $ or conformally equivalent to the metric $ g $. Throughout this article, we assume that $ \bar{M} $ is connected since otherwise we can easily apply arguments below equally to each connected component. It is well-known that this problem is reduced to the existence of the positive solutions of the nonlinear second order elliptic PDE
\begin{equation}\label{intro:eqn1}
-a\Delta_{g} u + R_{g} u = \left( S \circ \phi \right) u^{p-1} \; {\rm in} \; M, \frac{\partial u}{\partial \nu} + \frac{2}{p-2} h_{g} u = \frac{2}{p-2} \left( H \circ \phi \right) u^{\frac{p}{2}} \; {\rm on} \; \partial M.
\end{equation}
Here $ R_{g} $ is the scalar curvature of the metric $ g $, $ h_{g} $ is the mean curvature. $ \phi : \bar{M} \rightarrow \bar{M} $ is some diffeomorphism on $ \bar{M} $. When $ \phi = Id $, the PDE (\ref{intro:eqn1}) is for prescribing functions $ S, H $ within a conformal class $ [g] $. The constants $ a, p $ are defined as
\begin{equation*}
a = \frac{4(n - 1)}{n - 2}, p = \frac{2n}{n - 2}.
\end{equation*}
$ \Delta_{g} $ is the Laplace-Beltrami operator and $ \nu $ is the unique outward unit normal vector field along $ \partial M $. The functions $ S : \calC^{\infty}(\bar{M}) \rightarrow \R $, and $ H : \calC^{\infty}(\partial M) \rightarrow \R $ are given. We denote $ \eta_{1} $ to be the first eigenvalue of the conformal Laplacian $ \Box_{g} : = -a\Delta_{g} + R_{g} $ with associated eigenfunction $ \varphi $, i.e. $ \varphi $ is a positive, smooth function that solves the following PDE:
\begin{equation*}
-a\Delta_{g} \varphi + R_{g} \varphi = \eta_{1} \varphi \; {\rm in} \; M, \frac{\partial \varphi}{\partial \nu} + \frac{2}{p-2} h_{g} \varphi = 0 \; {\rm on} \; \partial M.
\end{equation*}
\medskip

When the dimension of the manifold $ n = 2 $, we also discuss the pair of functions $ K, \sigma $ that can be realized as Gaussian and geodesic curvatures, respectively, either for a pointwise conformal metric or a conformally equivalent metric. The two dimensional case is reduced to the existence of the solutions of the following elliptic PDE
\begin{equation}\label{intro:eqn2}
-a\Delta_{g} u + K_{g} = \left( K \circ \phi \right) e^{2u} \; {\rm in} \; M, \frac{\partial u}{\partial \nu} + \sigma_{g} = \left( \sigma \circ \phi \right) e^{u} \; {\rm on} \; \partial M.
\end{equation}
Here $ K_{g} $ and $ \sigma_{g} $ are Gaussian and geodesic curvatures of $ g $, respectively. The functions $ K :  \calC^{\infty}(\bar{M}) \rightarrow \R $ and $ \sigma : \calC^{\infty}(\partial M) \rightarrow \R $ are given. Again when the diffeomorphism $ \phi : \bar{M} \rightarrow \bar{M} $ is the identity map, $ K , \sigma $ are prescribing Gauss and geodesic curvatures for some metric within the conformal class $ [g] $.

% First Main result.
The main results of this article are given as follows:
\begin{theorem}\label{intro:thm1}
Let $ (\bar{M}, g) $ be a connected, compact manifold with non-empty smooth boundary $ \partial M $, $ n = \dim \bar{M} \geqslant 3 $. Let $ S, H \in \calC^{\infty}(\bar{M}) $ be given functions.
\begin{enumerate}[(i).]
\item If $ \eta_{1} < 0 $, then any function $ S < 0 $ somewhere in $ M $ can be realized as a scalar curvature function of some metric conformally equivalent to $ g $, with mean curvature $ cH $ for some small enough constant $ c > 0 $ and any function $ H $;
\item If $ \eta_{1} < 0 $, then any function $ S < 0 $ that changes sign in $ M $ can be realized as a scalar curvature function of some metric conformally equivalent to $ g $, with mean curvature $ cH $ for some small enough constant $ c > 0 $ and any function $ H $;
\item If $ \eta_{1} < 0 $, then any function $ S > 0 $ somewhere in $ M $ can be realized as a scalar curvature function of some metric pointwise conformal to $ g $, with mean curvature $ cH $ for some small enough constant $ c > 0 $ and any function $ H $.
\end{enumerate}
\end{theorem}
Case (i) is given in \S3 and \S4; when $ S < 0 $ everywhere on $ \bar{M} $, we can improve the result within a pointwise conformal class $ [g] $ in Theorem \ref{negative:thm1} and Theorem \ref{negative:thm2}; Case (ii) is given in \S6; when $ S $ satisfies $ \int_{M} S \dvol < 0 $ in addition, we can improve the result within a pointwise conformal class $ [g] $, see \cite[Thm~.1.2]{XU8}; and Case (iii) is given in \S7. The significance of this is that we can choose arbitrary function with small enough sup-norm as our mean curvature function, provided that the scalar curvature function is nontrivial. 
\medskip

Based on our best understanding, known results in this topic are mainly for the non-positive first eigenvalue cases or non-positive Euler characteristic cases. In \cite{CMR}, Cruz-Bl\'azquez, Malchiodi and Ruiz discussed prescribing negative scalar functions and mean curvature functions with arbitrary signs by variational method, for compact manifolds with dimensions at least $ 2 $. Some of our results overlap their results, but with a different method and different hypotheses on prescribed functions. However, our results are classified by the sign of the first eigenvalue of the conformal Laplacian. For zero first eigenvalue case or zero Euler characteristic case, we follow the results of \cite{XU8}. We point out that Brezis and Merle discussed the PDE $ -\Delta_{e} u = V e^{u} $ on $ \Omega \subset \R^{2} $ with Dirichlet boundaray condition in \cite{BreM}. Other results for the local Yamabe equation with Dirichlet condition in higher dimensions could be found in \cite{XU2}. For more discussions with respect to (\ref{intro:eqn2}) in $ 2 $-dimensional case, we refer to \cite[Ch.~13, Ch.~14]{T3}. When the first eigenvalue of conformal Laplacian is positive, a lot of non-existence results are given, e.g. \cite{BM} and \cite{MaMa}, etc.
\medskip

% The second main result.
We also give results on compact Riemann surfaces with boundary, provided that $ \chi(\bar{M}) < 0 $.
\begin{theorem}\label{intro:thm2}
Let $ (\bar{M}, g) $ be a compact Riemann surface with non-empty smooth boundary $ \partial M $. Let $ K, \sigma \in \calC^{\infty}(\bar{M}) $ be given functions. 
\begin{enumerate}[(i).]
\item If $ K < 0 $ everywhere on $ \bar{M} $, then there exists a metric pointwise conformal to $ g $ with Gauss curvature $ K $ and geodesic curvature $ c \sigma $ for some small enough constant $ c > 0 $ and arbitrary function $ \sigma $;
\item If $ K < 0 $ somewhere on $ \bar{M} $, then there exists a metric conformally equivalent to $ g $ with Gauss curvature $ K $ and geodesic curvature $ c \sigma $ for some small enough constant $ c > 0 $ and arbitrary function $ \sigma $.
\end{enumerate}
\end{theorem}
Both results above are given in \S5. Other than the related works on compact Riemann surfaces with boundary we introduced above, many work has been done on closed Riemann surface, a comprehensive study was given by Kazdan and Warner \cite{KW2}, including results for all signs of $ \chi(\bar{M}) $. For Nirenberg problem, we refer to Chang and Yang \cite{CY} and Struwe \cite{STRUWE}, etc.. 
\medskip

The most common method in analyzing this type of Kazdan-Warner problem is by calculus of variations since we can consider the PDE as Euler-Lagrange equation with respect to some functional; recently Morse theory is also involved. However, a new method, inspired by Kazdan and Warner \cite{KW}, has been developed recently. This new method applies monotone iteration scheme, a local version of calculus of variation to classify the existence results by sign of the first eigenvalue $ \eta_{1} $ of conformal Laplacian. This method has been applied to completely solve the Escobar problem \cite{XU4}, the Han-Li conjecture \cite{XU5}, the prescribed scalar curvature problem on compact manifolds \cite{XU6}, a trichotomy theorem in terms of prescribed scalar curvature with Dirichlet condition at boundary \cite{XU7}, and a comprehensive study of zero first eigenvalue case on compact manifolds, possibly with boundary, with dimensions at least $ 3 $ \cite{XU8}. In this article, we apply a variation of the combination of monotone iteration scheme and local analysis to show the results of prescribed scalar and mean curvatures for the cases $ \eta_{1} > 0 $ and $ \eta_{1} < 0 $. We also develop a general monotone iteration scheme, which can handle nonlinear terms both in the PDE and on the boundary condition; this new monotone iteration scheme, see Theorem \ref{HL:thm3}, allows us to work on $ 2 $-dimensional case without using the calculus of variation. This systematic procedure is powerful, but unfortunately this direct method cannot be used to the classical manifold, the unit ball with spherical boundary. We will explain why this direct method does not work in this case. Note that Escobar \cite{ESC3} has proved a nontrivial Kazdan-Warner type obstruction of prescribed mean curvature functions for this case.
\medskip

This paper is organized as follows:

In \S2, we introduce the essential definitions and results that will be used throughout this article. We assume the backgrounds of standard elliptic theory. We also introduced two versions of monotone iteration schemes. Theorem \ref{HL:thm2} is for the PDE (\ref{intro:eqn1}); Theorem \ref{HL:thm3} is more general, works for all second order semi-linear elliptic PDE with Robin boundary conditions, possibly with zeroth order nonlinear term on the boundary condition. Theorem \ref{HL:thm3} works well for the PDE like (\ref{intro:eqn2}). 

In \S3, we give results for prescribing scalar curvature function $ S $ and mean curvature function $ H $ within a conformal class $ [g] $ on $ (\bar{M}, g) $, $ n = \dim \bar{M} \geqslant 3 $, provided that $ \eta_{1} < 0 $. When $ S < 0 $ everywhere and arbitrary $ H $, the results are given in Theorem \ref{negative:thm1} and Theorem \ref{negative:thm2}. When $ S < 0 $ somewhere and arbitrary $ H $, the results are given in Theorem \ref{negative:thm3} and Corollary \ref{negative:cor1} with some restriction on $ S $. The monotone iteration scheme plays a central role.

In \S4, we give results for prescribing scalar curvature function $ S $ and mean curvature function $ H $ for some metric conformally equivalent to $ g $ on $ (\bar{M}, g) $, $ n = \dim \bar{M} \geqslant 3 $, provided that $ \eta_{1} < 0 $. It follows from Corollary \ref{negative:cor1}. We conclude in Theorem 4.1 that any $ S $ that is negative somewhere can be realized as a scalar curvature function of some metric conformally equivalent to $ g $, with the mean curvature $ cH $ for small enough constant $ c > 0 $ and arbitrary $ H $.

In \S5, we discuss prescribing Gauss and geodesic curvature functions $ K, \sigma $ on compact Riemann surfaces with boundary for metrics conformally equivalent to the original metric $ g $. We show in Theorem \ref{ne2:thm1} that any function $ K $ that is negative somewhere and satisfies some analytic condition can be realized as Gaussian curvature function for some metric conformally equivalent to $ g $, the metric also has geodesic curvature $ c\sigma $ for some small enough constant $ c > 0 $ and arbitrary $ \sigma $. The result in Corollary \ref{ne2:cor1} says that when $ K < 0 $ everywhere on $ \bar{M} $, the metric can be chosen within a conformal class $ [g] $. 

In \S6, we give results for prescribing scalar function $ S $ and mean curvature function $ H $ for some metric conformally equivalent to $ g $ on $ (\bar{M}, g) $, $ n = \dim \bar{M} \geqslant 3 $, provided that $ \eta_{1} = 0 $. We show that any function $ S $ that changes sign can be realized as a scalar curvature function some metric conformally equivalent to $ g $, with the mean curvature $ cH $ for small enough constant $ c > 0 $ and arbitrary $ H $ in Corollary \ref{zero:thm2}. Obviously there is a trivial case $ S \equiv H \equiv 0 $.

In \S7, we consider the prescribing scalar and mean curvature problem for $ \eta_{1} > 0 $. The results in Theorem \ref{positive:thm1} and Theorem \ref{positive:thm2} are for the case $ S > 0 $ somewhere and arbitrary $ H $. We also explain why our method cannot work on closed Euclidean ball with some nontrivial mean curvature on the boundary $ \mathbb{S}^{n} $.

% The Preliminaries and The Monotone Iteration Scheme and The Generalization of Han-Li Conjecture.
\section{The Preliminaries and The Monotone Iteration Scheme}
In this section, we first introduce the necessary definitions and essential results we need for the later sections, then introduce a general version of the monotone iteration scheme given in \cite{XU8}, other than many variations we have used in \cite{XU4, XU5, XU6, XU7, XU3}, with respect to the following Yamabe equation with Robin boundary condition
\begin{equation}\label{HL:eqn1}
-a\Delta_{g} u + R_{g} u = Su^{p-1} \; {\rm in} \; M, \frac{\partial u}{\partial \nu} + \frac{2}{p-2} h_{g} u = \frac{2}{p-2} H u^{\frac{p}{2}} \; {\rm on} \; \partial M.
\end{equation}
for given functions $ S, H \in \calC^{\infty}(\bar{M}) $, and $ n = \dim \bar{M} \geqslant 3 $. Lastly we introduce a $ W^{s, q} $-type regularity for elliptic PDE with Robin boundary conditions.

First of all, we give definitions of Sobolev spaces, a local version and a global version. Let $ \Omega $ be a connected, bounded, open subset of $ \R^{n} $ with smooth boundary $ \partial \Omega $ equipped with some Riemannian metric $ g $ that can be extended smoothly to $ \bar{\Omega} $. We call $ (\Omega, g) $ a Riemannian domain. Throughout this article, we denote the space of smooth functions with compact support by $ \calC_{c}^{\infty} $, smooth functions by $ \calC^{\infty} $, and continuous functions by $ \calC^{0} $.
% Sobolev Spaces.
\begin{definition}\label{HL:def1} Let $ (\Omega, g) $ be a Riemannian domain. Let $ (M, g) $ be a closed Riemannian $ n $-manifold, and $ (\bar{M}, g) $ be a compact Riemannian $n$-manifold with non-empty smooth boundary, with volume density $\dvol$. Let $u$ be a real valued function. Let $ \langle v,w \rangle_g$ and $ |v|_g = \langle v,v \rangle_g^{1/2} $ denote the inner product and norm  with respect to $g$. 

(i) 
For $1 \leqslant p < \infty $, we define the Lebesgue spaces on $ \Omega $ and $ \bar{M} $ to be
\begin{align*}
\mathcal{L}^{p}(\Omega)\ &{\rm is\ the\ completion\ of}\  \left\{ u \in \calC_c^{\infty}(\Omega) : \Vert u\Vert_p^p :=\int_{\Omega} \lvert u \rvert^{p} dx < \infty \right\},\\
\mathcal{L}^{p}(\Omega, g)\ &{\rm is\ the\ completion\ of}\ \left\{ u \in \calC_c^{\infty}(\Omega) : \Vert u\Vert_{p,g}^p :=\int_{\Omega} \left\lvert u \right\rvert^{p} d\text{Vol}_{g} < \infty \right\}, \\
\mathcal{L}^{p}(M, g)\ &{\rm is\ the\ completion\ of}\ \left\{ u \in \calC^{\infty}(M) : \Vert u\Vert_{p,g}^p :=\int_{M} \left\lvert u \right\rvert^{p} d\text{Vol}_{g} < \infty \right\}.
\end{align*}

(ii) For $\nabla u$  the Levi-Civita connection of $g$, 
and for $ u \in \calC^{\infty}(\Omega) $ or $ u \in \calC^{\infty}(\bar{M}) $,
\begin{equation}\label{HL:eqn2}
\lvert \nabla^{k} u \rvert_g^{2} := (\nabla^{\alpha_{1}} \dotso \nabla^{\alpha_{k}}u)( \nabla_{\alpha_{1}} \dotso \nabla_{\alpha_{k}} u).
\end{equation}
\noindent In particular, $ \lvert \nabla^{0} u \rvert^{2}_g = \lvert u \rvert^{2} $ and $ \lvert \nabla^{1} u \rvert^{2}_g = \lvert \nabla u \rvert_{g}^{2}.$\\

(iii) For $ s \in \mathbb{N}, 1 \leqslant p < \infty $, we define the $ (s, p) $-type Sobolev spaces on $ \Omega $ and $ \bar{M} $ to be
\begin{align}\label{HL:eqn3}
W^{s, p}(\Omega) &= \left\{ u \in \mathcal{L}^{p}(\Omega) : \lVert u \rVert_{W^{s,p}(\Omega)}^{p} : = \int_{\Omega} \sum_{j=0}^{s} \left\lvert D^{j}u \right\rvert^{p} dx < \infty \right\}, \\
W^{s, p}(\Omega, g) &= \left\{ u \in \mathcal{L}^{p}(\Omega, g) : \lVert u \rVert_{W^{s, p}(\Omega, g)}^{p} = \sum_{j=0}^{s} \int_{\Omega} \left\lvert \nabla^{j} u \right\rvert^{p}_g d\text{Vol}_{g} < \infty \right\} \nonumber, \\
W^{s, p}(M, g) &= \left\{ u \in \mathcal{L}^{p}(M, g) : \lVert u \rVert_{W^{s, p}(M, g)}^{p} = \sum_{j=0}^{s} \int_{M} \left\lvert \nabla^{j} u \right\rvert^{p}_g \dvol < \infty \right\} \nonumber.
\end{align}
\noindent Here $ \lvert D^{j}u \rvert^{p} := \sum_{\lvert \alpha \rvert = j} \lvert \partial^{\alpha} u \rvert^{p} $ 
in the weak sense. Similarly, $ W_{0}^{s, p}(\Omega) $ is the completion of $ \calC_{c}^{\infty}(\Omega) $ with respect to the 
$ W^{s, p} $-norm.
In particular, $ H^{s}(\Omega) : = W^{s, 2}(\Omega) $ and $ H^{s}(\Omega, g) : = W^{s, 2}(\Omega, g) $, $ H^{s}(M, g) : = W^{s, 2}(M, g) $ are the usual Sobolev spaces. We similarly define $H_{0}^{s}(\Omega), H_{0}^{s}(\Omega,g)$ and $ H_{0}^{s}(M, g) $.

(iv) On closed manifolds $ (M, g) $, we say that a function $ u \in  H^{s}(M, g) $ if $ u \in \calL^{2}(M, g) $ , and for any coordinate chart $ U \subset M $, any $ \psi \in \calC_{c}^{\infty}(U) $, the function $ \psi u \in H^{s}(U, g) $.
\end{definition}
\medskip
We assume the background of the standard elliptic theory, including the solvability of standard linear elliptic PDEs, elliptic regularity of $ H^{s} $-type, trace theorem, Sobolev embedding, Schauder estimates, etc. We introduce a $ W^{s, q} $-type elliptic regularity for later use.
\begin{theorem}\label{HL:thm1}\cite[Thm.~2.2]{XU5} Let $ (\bar{M}, g) $ be a compact manifold with smooth boundary $ \partial M $. Let $ \nu $ be the unit outward normal vector along $ \partial M $ and $ q > n = \dim \bar{M} $. Let $ L: \calC^{\infty}(\bar{M}) \rightarrow \calC^{\infty}(\bar{M}) $ be a uniform second order elliptic operator on $ M $ with smooth coefficients up to $ \partial M $ and can be extended to $ L : W^{2, q}(M, g) \rightarrow \calL^{q}(M, g) $. Let $ f \in \calL^{q}(M, g), \tilde{f} \in W^{1, q}(M, g) $. Let $ u \in H^{1}(M, g) $ be a weak solution of the following boundary value problem
\begin{equation}\label{HL:eqn4}
L u = f \; {\rm in} \; M, Bu = \frac{\partial u}{\partial \nu} + c(x) u = \tilde{f} \; {\rm on} \; \partial M.
\end{equation}
Here $ c \in \calC^{\infty}(M) $. Assume also that $ \text{Ker}(L) = \lbrace 0 \rbrace $ associated with the homogeneous Robin boundary condition. If, in addition, $ u \in \calL^{q}(M, g) $, then $ u \in W^{2, q}(M, g) $ with the following estimates
\begin{equation}\label{HL:eqn5}
\lVert u \rVert_{W^{2, q}(M, g)} \leqslant \gamma' \left(\lVert Lu \rVert_{\calL^{q}(M, g)} + \lVert Bu \rVert_{W^{1, q}(M, g)} \right)
\end{equation}
Here $ \gamma' $ depends on $ L, q, c $ and the manifold $ (\bar{M}, g) $ and is independent of $ u $.
\end{theorem}

We then introduce the first eigenvalue of conformal Laplacian. Note that $ a = \frac{4(n - 1)}{n - 2} $ and $ p = \frac{2n}{n - 2} $, hence it only makes sense when $ n \geqslant 3 $.
% Conformal Laplacian, First Eigenvalue
\begin{definition}\label{HL:def2}
Let $ (\bar{M}, g) $ be a compact manifold with non-empty smooth boundary $ \partial M $. We denote $ \eta_{1} $ be the first eigenvalue of conformal Laplacian with its corresponding eigenfunction $ \varphi > 0 $ if and only if the following PDE holds.
\begin{equation}\label{HL:eqn6}
-a\Delta_{g} \varphi + R_{g} \varphi = \eta_{1} \varphi \; {\rm in} \; M, \frac{\partial \varphi}{\partial \nu} + \frac{2}{p-2} h_{g} \varphi = 0 \; {\rm on} \; \partial M.
\end{equation}
\end{definition}

% Monotone iteration scheme.
We now introduce a variation of the monotone iteration scheme we used in \cite{XU4}, \cite{XU5} and \cite{XU6}. In particular, we do require $ h_{g} = h > 0 $ to be some positive constant on $ \partial M $, this can be done due to the proof of the Han-Li conjecture in \cite{XU5}. We will also use other versions of monotone iteration schemes introduced in eariler work \cite{XU4, XU5, XU6, XU7, XU3}.
\begin{theorem}\label{HL:thm2}\cite[Thm.~2.4]{XU8}
Let $ (\bar{M}, g) $ be a compact manifold with smooth boundary $ \partial M $. Let $ \nu $ be the unit outward normal vector along $ \partial M $ and $ q > \dim \bar{M} $. Let $ S \in \calC^{\infty}(\bar{M}) $ and $ H \in \calC^{\infty}(\bar{M}) $ be given functions. Let the mean curvature $ h_{g} = h > 0 $ be some positive constant. In addition, we assume that $ \sup_{\bar{M}} \lvert H \rvert $ is small enough. Suppose that there exist $ u_{-} \in \calC_{0}(\bar{M}) \cap H^{1}(M, g) $ and $ u_{+} \in W^{2, q}(M, g) \cap \calC_{0}(\bar{M}) $, $ 0 \leqslant u_{-} \leqslant u_{+} $, $ u_{-} \not\equiv 0 $ on $ \bar{M} $, some constants $ \theta_{1} \leqslant 0, \theta_{2} \geqslant 0 $ such that
\begin{equation}\label{HL:eqn7}
\begin{split}
-a\Delta_{g} u_{-} + R_{g} u_{-} - S u_{-}^{p-1} & \leqslant 0 \; {\rm in} \; M, \frac{\partial u_{-}}{\partial \nu} + \frac{2}{p-2} h_{g} u_{-} \leqslant \theta_{1} u_{-} \leqslant \frac{2}{p-2} H u_{-}^{\frac{p}{2}} \; {\rm on} \; \partial M \\
-a\Delta_{g} u_{+} + R_{g} u_{+} - S u_{+}^{p-1} & \geqslant 0 \; {\rm in} \; M, \frac{\partial u_{+}}{\partial \nu} + \frac{2}{p-2} h_{g} u_{+} \geqslant \theta_{2} u_{+} \geqslant \frac{2}{p-2} H u_{+}^{\frac{p}{2}} \; {\rm on} \; \partial M
\end{split}
\end{equation}
holds weakly. In particular, $ \theta_{1} $ can be zero if $ H \geqslant 0 $ on $ \partial M $, and $ \theta_{1} $ must be negative if $ H < 0 $ somewhere on $ \partial M $; similarly, $ \theta_{2} $ can be zero if $ H \leqslant 0 $ on $ \partial M $, and $ \theta_{2} $ must be positive if $ H > 0 $ somewhere on $ \partial M $. Then there exists a real, positive solution $ u \in \calC^{\infty}(M) \cap \calC^{1, \alpha}(\bar{M}) $ of
\begin{equation}\label{HL:eqn8}
\Box_{g} u = -a\Delta_{g} u + R_{g} u = S u^{p-1}  \; {\rm in} \; M, B_{g} u =  \frac{\partial u}{\partial \nu} + \frac{2}{p-2} h_{g} u = \frac{2}{p-2} H u^{\frac{p}{2}} \; {\rm on} \; \partial M.
\end{equation}
\end{theorem}
\medskip

The following two results are necessary, which shows the existence of the solution of some local Yamabe-type problem. When the manifold is not locally conformally flat, we need
% Local solution of the Yamabe equation.
\begin{proposition}\label{HL:prop1}\cite[Prop.~3.2]{XU8}
Let $ (\Omega, g) $ be a Riemannian domain in $\R^n$, $ n \geqslant 3 $, not locally conformally flat, with $C^{\infty} $ boundary, with ${\rm Vol}_g(\Omega)$ and the Euclidean diameter of $\Omega$ sufficiently small. Let $ f \in \Omega' \supset \Omega $ be a positive, smooth function in some open region $ \Omega' $. In addition, we assume that the first eigenvalue of Laplace-Beltrami operator $ -\Delta_{g} $ on $ \Omega $ with Dirichlet condition satisfies $ \lambda_{1} \rightarrow \infty $ as $ \Omega $ shrinks. Assume $ R_{g} < 0 $ within the small enough closed domain $ \bar{\Omega} $. Then the Dirichlet problem
\begin{equation}\label{HL:eqn9}
-a\Delta_{g} u + R_{g} u = f u^{p-1} \; {\rm in} \; \Omega, u \equiv 0 \; {\rm on} \; \partial \Omega
\end{equation}
 has a real, positive, smooth solution $ u \in \calC^{\infty}(\Omega) \cap H_{0}^{1}(\Omega, g) \cap \calC^{0}(\bar{\Omega}) $. The size of $ \Omega $ is depending on the function $ f $.
\end{proposition}
When the manifold is locally conformally flat, we give the local solution of (\ref{HL:eqn9}) provided that $ \Omega $ is not topologically trivial.
\begin{proposition}\label{HL:prop2}\cite[Prop.~2.5]{XU6}
Let $ (\Omega, g) $ be a Riemannian domain in $\R^n$, $ n \geqslant 3 $, with $C^{\infty} $ boundary. Let the metric $ g $ be locally conformally flat on some open subset $ \Omega' \supset \bar{\Omega} $. For any point $ \rho \in \Omega $ and any positive constant $ \epsilon $, denote the region $ \Omega_{\epsilon} $ to be
\begin{equation*}
\Omega_{\epsilon} = \lbrace x \in \Omega | \lvert x - \rho \rvert > \epsilon \rbrace.
\end{equation*}
Assume that $ Q \in \calC^{2}(\bar{\Omega}) $, $ \min_{x \in \bar{\Omega}} Q(x) > 0 $ and $ \nabla Q(\rho) \neq 0 $. Then there exists some $ \epsilon_{0} $ such that for every $ \epsilon \in (0, \epsilon_{0}) $ the Dirichlet problem
\begin{equation}\label{HL:eqn10}
-a\Delta_{g}u + R_{g} u = Qu^{p-1} \; {\rm in} \; \Omega_{\epsilon}, u = 0 \; {\rm on} \; \partial \Omega_{\epsilon}
\end{equation}
has a real, positive, smooth solution $ u \in \calC^{\infty}(\Omega_{\epsilon}) \cap H_{0}^{1}(\Omega_{\epsilon}, g) \cap \calC^{0}(\bar{\Omega_{\epsilon}}) $.
\end{proposition}
\begin{remark}\label{HL:re1}
It is straightforward to see that under conformal change $ \tilde{g} = \phi^{p-2} g $, we have
\begin{equation}\label{HL:eqn11}
\tilde{g} = \phi^{p-2} g \Rightarrow -a\Delta_{\tilde{g}} + R_{\tilde{g}} = \phi^{-\frac{n+2}{n-2}} \left(-a\Delta_{g} + R_{g} \right) \phi \Leftrightarrow \Box_{\tilde{g}} = \phi^{1 - p} \Box_{g} \phi.
\end{equation}
We call (\ref{HL:eqn11}) the conformal invariance of the conformal Laplacian. It follows from Proposition \ref{HL:prop2} and (\ref{HL:eqn11}) that if the manifold $ (\bar{M}, g) $ is locally conformally flat in the interior, the equation (\ref{HL:eqn10}) is equivalent to \begin{equation}\label{HL:eqn12}
-a\Delta_{g_{e}} u = Q u^{p-1} \; {\rm in} \; \Omega_{\epsilon}, u = 0 \; {\rm on} \; \partial \Omega_{\epsilon}
\end{equation}
which admits a positive solution $ u \in \calC^{\infty}(\Omega_{\epsilon}) \cap H_{0}^{1}(\Omega_{\epsilon}, g) \cap \calC^{0}(\bar{\Omega_{\epsilon}}) $.
\end{remark}
\medskip

As a prerequisite, we also need a result in terms of the perturbation of negative first eigenvalue of conformal Laplacian.
\begin{proposition}\label{HL:prop3}
Let $ (\bar{M}, g) $ be a compact Riemannian manifold with non-empty smooth boundary $ \partial M $, $ n = \dim \bar{M} \geqslant 3 $. Let $ \beta > 0 $ be a small enough constant. If $ \eta_{1}' < 0 $, then the quantity
\begin{equation*}
\eta_{1, \beta}' = \inf_{u \neq 0} \frac{a\int_{M} \lvert \nabla_{g} u \rvert^{2} \dvol + \int_{M} R_{g} u^{2} \dvol + \frac{2a}{p-2} \int_{\partial M} (h_{g} + \beta) u^{2} dS}{\int_{M} u^{2} \dvol} < 0.
\end{equation*}
In particular, $ \eta_{1, \beta}' $ satisfies
\begin{equation}\label{HL:eqn12}
-a\Delta_{g} \varphi + R_{g} \varphi = \eta_{1, \beta}' \varphi \; {\rm in} \; M, \frac{\partial \varphi}{\partial \nu} + \frac{2}{p-2} (h_{g} + \beta) \varphi = 0 \; {\rm on} \; \partial M
\end{equation}
with some positive function $ \varphi \in \calC^{\infty}(\bar{M}) $.
\end{proposition}
\begin{proof}
Since $ \eta_{1}' < 0 $, the normalized first eigenfunction $ \varphi_{1} $, i.e. $ \int_{M} \varphi_{1}^{2} \dvol = 1 $, satisfies
\begin{equation*}
\eta_{1}' = a\int_{M} \lvert \nabla_{g} \varphi_{1} \rvert^{2} \dvol + \int_{M} R_{g} \varphi_{1}^{2} \dvol + \frac{2a}{p-2} \int_{\partial M} h_{g} \varphi_{1}^{2} dS
\end{equation*}
By characterization of $ \eta_{1, \beta}' $, we have
\begin{equation*}
\eta_{1, \beta}' \leqslant a\int_{M} \lvert \nabla_{g} \varphi_{1} \rvert^{2} \dvol + \int_{M} R_{g} \varphi_{1}^{2} \dvol + \frac{2a}{p-2} \int_{\partial M} (h_{g} + \beta) \varphi_{1}^{2} dS = \eta_{1}' + \beta \int_{\partial M} \varphi_{1}^{2} dS.
\end{equation*}
Since $ \varphi_{1} $ is fixed, it follows that $ \eta_{1, \beta}' < 0 $ if $ \beta > 0 $ is small enough. 
\end{proof}
\medskip

When $ n = 2 $, i.e. $ M $ or $ \bar{M} $ is a compact Riemann surface (possibly with boundary), all tools above are not available. We thus need a new version of the monotone iteration scheme for compact Riemann surfaces with non-empty smooth boundary. We point out that the monotone iteration scheme below works for all compact manifolds with non-empty boundary, with dimensions at least $ 2 $.
% Most General Version of The Monotone Iteration Scheme.
\begin{theorem}\label{HL:thm3}
Let $ (\bar{M}, g) $ be a compact manifold with non-empty smooth boundary $ \partial M $, $ n = \dim M \geqslant 2 $. Let $ q > n $ be a positive integer. Let $ F(\cdot, \cdot), G(\cdot, \cdot) : \bar{M} \times \R \rightarrow \R $ be smooth functions. Let $ \nu $ be the unit outward normal vector along $ \partial M $. Let $ \sigma $ be some nonnegative, small enough constant. If

(i) there exists two functions $ u_{+} \in \calC^{\infty}(\bar{M}) $ and $ u_{-} \in \calC^{0}(\bar{M}) \cap H^{1}(M, g) $ such that
\begin{equation}\label{HL:eqn13}
\begin{split}
-\Delta_{g} u_{+} & \geqslant F(\cdot, u) \; {\rm in} \; M, \frac{\partial u}{\partial \nu} + \sigma u \geqslant G(\cdot, u_{+}) \; {\rm on} \; \partial M; \\
-\Delta_{g} u_{-} & \leqslant F(\cdot, u) \; {\rm in} \; M, \frac{\partial u}{\partial \nu} + \sigma u \leqslant G(\cdot, u_{-}) \; {\rm on} \; \partial M,
\end{split}
\end{equation}
where the sub-solution may hold in the weak sense; and

(ii) in addition, $ \sup_{\bar{M}} \lvert G(\cdot, u_{+}) \rvert, \sup_{\bar{M}} \lvert \nabla G(\cdot, u_{+}) \rvert $ are small enough;

(iii) furthermore, $ u_{+} \geqslant u_{-} $ pointwise on $ \bar{M} $;

then there exists a smooth function $ u \in \calC^{\infty}(\bar{M}) $ with $ u_{-} \leqslant u \leqslant u_{+} $ such that
\begin{equation}\label{HL:eqn14}
-\Delta_{g} u = F(\cdot, u) \; {\rm in} \; M, \frac{\partial u}{\partial \nu} + \sigma u = G(\cdot, u) \; {\rm on} \; \partial M.
\end{equation}
\end{theorem}
\begin{remark}\label{HL:re2}
The proof of Theorem \ref{HL:thm3} is essentially the same as the proof of \cite[Thm.~2.4]{XU8}, except some minor change, e.g. here we use general smooth functions $ F $ and $ G $ but not specific Yamabe equations. We therefore will give a relatively concise proof for Theorem \ref{HL:thm3}.
\end{remark}
\begin{proof}$ \bar{M} $ is compact, so extremal values of continuous functions $ u_{+}, u_{-} $ can be achieved. Choose positive constant $ A $ and nonnegative constant $ B $ such that
\begin{equation}\label{HL:eqn15}
\begin{split}
A & \geqslant -\frac{\partial F}{\partial u}(x, u(x)), \forall x \in \bar{M}, u(x) \in [\min_{\bar{M}} u_{-}, \max_{\bar{M}} u_{+}]; \\
B & \geqslant \sigma - \frac{\partial G}{\partial u}(x, u(x)), \forall x \in \bar{M}, u(x) \in [\min_{\bar{M}} u_{-}, \max_{\bar{M}} u_{+}].
\end{split}
\end{equation}
Denote $ u_{0} = u_{+} \in \calC^{\infty}(\bar{M}) $, and consider the iteration scheme
\begin{equation}\label{HL:eqn16}
\begin{split}
-\Delta_{g} u_{k} + A u_{k} & = A u_{k - 1} + F(\cdot, u_{k - 1}) \; {\rm in} \; M, k \in \mathbb{N}, \\
\frac{\partial u_{k}}{\partial \nu} + B u_{k} & = Bu_{k - 1} - \sigma u_{k - 1} + G(\cdot, u_{k - 1}) \; {\rm on} \; \partial M, k \in \mathbb{N}.
\end{split}
\end{equation}
Since $ A > 0, B \geqslant 0 $, the operator
\begin{equation*}
\left( -\Delta_{g} + A, \frac{\partial}{\partial \nu} + B \right)
\end{equation*}
is invertible due to the standard argument. Clearly when $ k = 1 $, the first iteration step in (\ref{HL:eqn16}) gives a unique smooth solution $ u_{1} \in \calC^{\infty}(\bar{M}) $. The regularity argument is also standard.

We show that $ u_{-} \leqslant u \leqslant u_{+} $. For $ u_{-} \leqslant u $, we have to use the sub-solution in the weak sense, since $ u_{0} = u_{+} \geqslant u_{-} $, we pair (\ref{HL:eqn16}) for $ k = 1 $ with arbitrary non-negative function $ v \in \calC^{\infty}(\bar{M}) $, and subtract this with the sub-solution (adding $ Au_{-} $ and $ Bu_{-} $ on both sides of the PDE and boundary conditions respectively) in the weak sense, we have
\begin{align*}
& \int_{M} \left( A \left( u_{0} - u_{-} \right) + F\left(x, u_{0}\right) - F\left(x, u_{-}\right) \right) v \dvol \leqslant \int_{M} \left( -\Delta_{g} \left( u_{1} - u_{-} \right) + A \left( u_{1} - u_{-} \right) \right) v \dvol \\
\leqslant & \int_{\partial M} B\left( u_{1} - u_{-} \right) v dS  - \int_{\partial M} \left( B \left( u_{0} - u_{-} \right) - \sigma \left( u_{0} - u_{-} \right) + G\left(x, u_{0}\right) - G\left(x, u_{-}\right) \right) v dS  \\
& \qquad + \int_{M} A \left( u_{1} - u_{-} \right) v \dvol + \int_{M} \nabla_{g} \left( u_{1} - u_{-} \right) \cdot \nabla_{g} v \dvol.
\end{align*}
Taking $ v= w : = \max \left( u_{-} - u_{1}, 0 \right) $, and applying the mean value theorem for $ F, G $, due to the definitions of $ A, B $ in (\ref{HL:eqn15}), we observe that
\begin{equation*}
\int_{M} \lvert \nabla_{g} w \rvert^{2} + \int_{\partial M} Bw^{2} + \int_{M} A w^{2} \leqslant 0.
\end{equation*}
It follows that $ w = 0 $, therefore $ u_{-} \leqslant u_{1} $. By a very similar argument in terms of the subtraction between (\ref{HL:eqn16}) and the super-solution, we conclude that $ u_{+} \geqslant u_{1} $.

Inductively, we may assume the existence of the solutions $ u_{1}, \dotso, u_{k} $ with
\begin{equation*}
u_{-} \leqslant u_{k} \leqslant u_{k-1} \leqslant \dotso \leqslant u_{1} \leqslant u_{0}.
\end{equation*}
By the same argument in the first iteration step, we conclude the existence of $ u_{k+1} \in \calC^{\infty}(\bar{M}) $; in addition, $ u_{k+1} $ satisfies
\begin{equation*}
u_{-} \leqslant u_{k+1} \leqslant u_{k} \leqslant u_{k-1} \leqslant \dotso \leqslant u_{1} \leqslant u_{0}.
\end{equation*}
Therefore we show the existence of the sequence of solutions of (\ref{HL:eqn16}) with the monotonicity
\begin{equation}\label{HL:eqn17}
u_{-} \leqslant \dotso \leqslant u_{k+1} \leqslant u_{k} \leqslant u_{k-1} \leqslant \dotso \leqslant u_{0}, k \in \mathbb{N}.
\end{equation}

We now show the uniform boundedness of $ \lVert u_{k} \rVert_{\calC^{1, \alpha}(\bar{M})} $. Since $ q > n $, showing the uniform boundedness of $ \lVert u_{k} \rVert_{\calC^{1, \alpha}(\bar{M})} $ is equivalent to show the uniform boundedness of $ \lVert u_{k} \rVert_{W^{2, q}(M, g)} $. We have mentioned that the operator is invertible and thus the $ W^{s, q} $-type estimates (\ref{HL:eqn5}) applies. We $ L $ and the boundary condition $ c $ to be the operators here with associated constant $ \gamma' $. Mimicking the boundedness proof in \cite[Thm.~2.4]{XU8}, we should require $ \sigma $ and $ \sup_{\bar{M}} \lvert G(\cdot, u) \rvert $, and $ \sup_{\bar{M}} \lvert \nabla G(\cdot, u) \rvert $ to be small enough. Denote
\begin{equation}\label{HL:eqn18}
\begin{split}
C & = \sup_{x \in \bar{M}, u(x) \in [\min_{\bar{M}} u_{-}, \max_{\bar{M}} u_{+}]} \lvert F(x, u(x)) \rvert; \\
D_{1} & = \sup_{x \in \bar{M}, u(x) \in [\min_{\bar{M}} u_{-}, \max_{\bar{M}} u_{+}]} \left\lvert G(x, u(x)) \right\rvert; \\
D_{2} & = \sup_{x \in \bar{M}, u(x) \in [\min_{\bar{M}} u_{-}, \max_{\bar{M}} u_{+}]} \left\lvert \nabla G(x, u(x)) \right\rvert; \\
\end{split}
\end{equation}
We require that $ G(\cdot, u_{+}), D_{1}, D_{2} $ satisfies
\begin{equation}\label{HL:eqn19}
\begin{split}
& \left\lVert \left( B - \sigma \right) u_{+} + G\left(\cdot, u_{+} \right) \right\rVert_{W^{1, q}(M, g)} \leqslant 1; \\
& \left( B - \sigma \right) \cdot \gamma' \left( \left( A \max_{\bar{M}} \left( \lvert u_{+} \rvert, \lvert u_{-} \rvert \right) + C \right) \cdot\text{Vol}_{g}(M)^{\frac{1}{q}} + 1 \right) + D_{1} \cdot \text{Vol}_{g}(M)^{\frac{1}{q}} \\
& \qquad + D_{2} \cdot \gamma' \left( \left( A \max_{\bar{M}} \left( \lvert u_{+} \rvert, \lvert u_{-} \rvert \right) + C \right) \cdot \text{Vol}_{g}(M)^{\frac{1}{q}} + 1 \right) \leqslant 1.
\end{split}
\end{equation}
By (\ref{HL:eqn5}) and the first inequality in (\ref{HL:eqn19}), we observe from the PDE (\ref{HL:eqn16}) with $ k = 1 $ that
\begin{align*}
\lVert u_{1} \rVert_{W^{2, q}(M, g)} & \leqslant \gamma' \left( \lVert A u_{+} + F(\cdot, u_{+}) \rVert_{\calL^{q}(M, g)} + \left\lVert \left( B - \sigma \right) u_{+} + G\left(\cdot, u_{+} \right) \right\rVert_{W^{1, q}(M, g)} \right) \\
& \leqslant \gamma' \left( \left( A \max_{\bar{M}} \lvert u_{+} \rvert + C \right) \cdot \text{Vol}_{g}(M)^{\frac{1}{q}} + 1 \right) \\
& \leqslant \gamma' \left( \left( A \max_{\bar{M}} \left( \lvert u_{+} \rvert, \lvert u_{-} \rvert \right) + C \right) \cdot \text{Vol}_{g}(M)^{\frac{1}{q}} + 1 \right).
\end{align*}
Inductively, assume that
\begin{equation}\label{HL:eqn20}
\lVert u_{k} \rVert_{W^{2, q}(M, g)} \leqslant \gamma' \left( \left( A \max_{\bar{M}} \left( \lvert u_{+} \rvert, \lvert u_{-} \rvert \right) + C \right) \cdot \text{Vol}_{g}(M)^{\frac{1}{q}} + 1 \right).
\end{equation}
To check $ \lVert u_{k+1} \rVert_{W^{2, q}(M, g)} $, we apply the $ W^{s, q} $-type elliptic estimate with the solution of (\ref{HL:eqn16}) again, 
\begin{align*}
\lVert u_{k+1} \rVert_{W^{2, q}(M, g)} & \leqslant \gamma' \left( \lVert A u_{k} + F(\cdot, u_{k}) \rVert_{\calL^{q}(M, g)} + \left\lVert \left( B - \sigma \right) u_{k} + G\left(\cdot, u_{k} \right) \right\rVert_{W^{1, q}(M, g)} \right) \\
& \leqslant \gamma' \left( \left( A \max_{\bar{M}} \left( \lvert u_{+} \rvert, \lvert u_{-} \rvert \right) + C \right) \cdot \text{Vol}_{g}(M)^{\frac{1}{q}} \right) \\
& \qquad + \gamma' \left( (B - \sigma) \lVert u_{k} \rVert_{W^{1, q}(M, g)} + \lVert G(\cdot, u_{k}) \rVert_{\calL^{q}(M, g)} + \lVert \nabla G(\cdot, u_{k}) \rVert_{\calL^{q}(M, g)} \right) \\
& \leqslant \gamma' \left( \left( A \max_{\bar{M}} \left( \lvert u_{+} \rvert, \lvert u_{-} \rvert \right) + C \right) \cdot \text{Vol}_{g}(M)^{\frac{1}{q}} \right) \\
& \qquad + \left( \gamma' \right)^{2} (B - \sigma) \left( A \max_{\bar{M}} \left( \left( \lvert u_{+} \rvert, \lvert u_{-} \rvert \right) + C \right) \cdot \text{Vol}_{g}(M)^{\frac{1}{q}} + 1 \right) \\
& \qquad \qquad + \gamma' D_{1} \cdot \text{Vol}_{g}(M)^{\frac{1}{q}} +\left( \gamma' \right)^{2} D_{2} \left( A \max_{\bar{M}} \left( \left( \lvert u_{+} \rvert, \lvert u_{-} \rvert \right) + C \right) \cdot \text{Vol}_{g}(M)^{\frac{1}{q}} + 1 \right) \\
& \leqslant \gamma' \left( \left( A \max_{\bar{M}} \left( \lvert u_{+} \rvert, \lvert u_{-} \rvert \right) + C \right) \cdot \text{Vol}_{g}(M)^{\frac{1}{q}} + 1 \right).
\end{align*}
It turns that $ \lVert u_{k} \rVert_{W^{2, q}(M, g)} $ is uniformly bounded. The rest of the argument, in applying Arzela-Ascoli, the monotonicity of the sequence, and the elliptic regularity, is essentially the same as in \cite[Thm.~2.4]{XU8}. We omit the details here.

In conclusion, the sequence $ u_{k} $ converges classically to a smooth function $ u $ which solves (\ref{HL:eqn14}). In addition, $ u_{-} \leqslant u \leqslant u_{+} $ pointwise on $ \bar{M} $.
\end{proof}
\begin{remark}\label{HL:re3}
Theorem \ref{HL:thm2} is a special case of Theorem \ref{HL:thm3} by taking $ F(\cdot, u) = - R_{g} u + Su^{p-1} $ and $ G(\cdot, u) = -\frac{2}{p-2} h_{g} u + \frac{2}{p-2} H u^{\frac{p}{2}} $.
\end{remark}
\medskip

% This is the section for prescribed scalar curvature and mean curvature when $ \eta_{1} < 0 $.
\section{Prescribed Scalar and Mean Curvature Functions under Pointwise Conformal Deformation When $ \eta_{1} < 0 $}
Recall the Yamabe equation with Robin condition
\begin{equation}\label{negative:eqn1}
-a\Delta_{g} u + R_{g} u = Su^{p-1} \; {\rm in} \; M, \frac{\partial u}{\partial \nu} + \frac{2}{p-2} h_{g} u = \frac{2}{p-2} H u^{\frac{p}{2}} \; {\rm on} \; \partial M.
\end{equation}
In this section, we consider the existence of the solution of (\ref{negative:eqn1}) for given functions $ S, H \in \calC^{\infty}(\bar{M}) $, provided that $ \eta_{1} < 0 $. In particular, we will discuss the following cases:
\begin{enumerate}[(i).]
\item $ S <  0 $ in $ M $, and $ H \leqslant 0 $ everywhere on $ \partial M $, $ H \not\equiv 0 $, with $ \eta_{1} < 0 $; 
\item $ S <  0 $ in $ M $, and $ H > 0 $ somewhere on $ \partial M $, with $ \eta_{1} < 0 $;
\item $ S $ changes sign in $ M $, and $ H $ is arbitrary on $ \partial M $, with $ \eta_{1} < 0 $.
\end{enumerate}
Note that the Case (ii) above covers the possibilities when $ H > 0 $ everywhere on $ \partial M $, or $ \int_{\partial M} H dS > 0 $. Note also that the case $ S < 0 $ everywhere in $ M $ and $ H = 0 $ on $ \partial M $ has been discussed in \cite{XU6}. For Case (iii), obviously we have to impose some restrictions on $ S $ and $ H $, as we shall see later; there is no free choice of $ S $ especially, due to Kazdan and Warner \cite{KW}. The first result concerns the Case (i).
% The case $ S < 0, H < 0 $ and $ \eta_{1} < 0 $.
\begin{theorem}\label{negative:thm1}
Let $ (\bar{M}, g) $ be a compact manifold with non-empty smooth boundary $ \partial M $, $ n = \dim \bar{M} \geqslant 3 $. Let $ S_{1} < 0 $ be any smooth function on $ \bar{M} $. Let $ H_{1} \in \calC^{\infty}(\bar{M}) $ such that $ H_{1} < 0 $ everywhere on $ \partial M $. If $ \eta_{1} < 0 $, then there exists a small enough constant $ c > 0 $ such that (\ref{negative:eqn1}) admits a positive solution $ u \in \calC^{\infty}(\bar{M}) $ with $ S = S_{1} $ and $ H = cH_{1} $. Equivalently, there exists a Yamabe metric $ \tilde{g} = u^{p-2} g $ such that $ R_{\tilde{g}} = S_{1} $ and $ h_{\tilde{g}} = cH_{1} \bigg|_{\partial M} $.
\end{theorem}
\begin{proof} Due to the proof of Han-Li conjecture \cite[Theorem]{XU5}, we may assume that $ h_{g} = h > 0 $ and $ R_{g} < 0 $. Since $ \eta_{1} < 0 $, it follows that $ \eta_{1, \beta} < 0 $ with small enough positive constant $ \beta > 0 $, due to Proposition \ref{HL:prop3}. Any constant multiple of $ \varphi $ solves (\ref{HL:eqn12}). Denote $ \phi = \delta \varphi $, we choose the constant $ \delta > 0 $ small enough so that
\begin{equation*}
\eta_{1, \beta} \inf_{\bar M} \varphi \geqslant \delta^{p-2} \cdot \inf_{\bar{M}} S_{1} \cdot  \sup_{\bar{M}} \varphi^{p-1}.
\end{equation*}
This can be done since both $ \eta_{1, \beta} $ and $ S_{1} $ are negative functions. It follows that
\begin{equation*}
-a\Delta_{g} \phi + R_{g} \phi = \eta_{1, \beta} \phi \leqslant S_{1} \phi^{p-1} \; {\rm in} \; M.
\end{equation*}
Fix this $ \delta $. We check the boundary condition
\begin{equation*}
-\frac{\partial \phi}{\partial \nu} + \frac{2}{p-2} h_{g} \phi = -\beta \cdot \frac{2}{p-2} \phi \leqslant \frac{2}{p-2} \cdot \left(cH_{1} \right) \phi^{\frac{p}{2}}
\end{equation*}
for small enough positive constant $ c > 0 $. Again it works since both $ -\beta $ and $ H_{1} $ are negative. We set
\begin{equation}\label{negative:eqn2}
u_{-} : = \phi.
\end{equation}
The argument above shows that $ u_{-} $ is a sub-solution of (\ref{negative:eqn1}) with $ S = S_{1} $ and $ H = cH_{1} $ for small enough $ c $. For super-solution, we set
\begin{equation}\label{negative:eqn3}
u_{+} : = C \gg 1.
\end{equation}
When $ C $ large enough, we have
\begin{equation*}
-a\Delta_{g} u_{+} + R_{g} u_{+} = R_{g} C \geqslant S_{1} C^{p-1} \; {\rm in} \; M.
\end{equation*}
Since $ H_{1} < 0 $, it is straightforward to check that for any $ c > 0 $, we have
\begin{equation*}
-\frac{\partial u_{+}}{\partial \nu} + \frac{2}{p-2} h_{g} u_{+} \geqslant 0 > \frac{2}{p-2} \left(cH_{1} \right) u_{+}^{\frac{p}{2}}.
\end{equation*}
We can enlarge $ C $ so that $ C \geqslant \sup_{\bar{M}} u_{-} $. Lastly we shrink $ c $ if necessary since we require the smallness of the sup-norm of the prescribing mean curvature function in the proof of Theorem \ref{HL:thm2}. Since $ 0 < u_{-} \leqslant u_{+} $ and both $ u_{+} $ and $ u_{-} $ are smooth functions, we conclude by Theorem \ref{HL:thm2} that (\ref{negative:eqn1}) has a positive solution $ u \in \calC^{\infty}(\bar{M}) $ with $ S = S_{1} $ and $ H = cH_{1} $ for small enough $ c > 0 $.
\end{proof}
We now consider the Case (ii) at the beginning of this section. Actually the proof is very similar to Theorem \ref{negative:thm1} above.
% The case $ S < 0 $, $ H > 0 $ somewhere and $ \eta_{1} < 0 $.
\begin{theorem}\label{negative:thm2}
Let $ (\bar{M}, g) $ be a compact manifold with non-empty smooth boundary $ \partial M $, $ n = \dim \bar{M} \geqslant 3 $. Let $ S_{2} < 0 $ be any smooth function on $ \bar{M} $. Let $ H_{2} \in \calC^{\infty}(\bar{M}) $ such that $ H_{2} > 0 $ somewhere on $ \partial M $. If $ \eta_{1} < 0 $, then there exists a small enough constant $ c > 0 $ such that (\ref{negative:eqn1}) admits a positive solution $ u \in \calC^{\infty}(\bar{M}) $ with $ S = S_{2} $ and $ H = cH_{2} $. Equivalently, there exists a Yamabe metric $ \tilde{g} = u^{p-2} g $ such that $ R_{\tilde{g}} = S_{2} $ and $ h_{\tilde{g}} = cH_{2} \bigg|_{\partial M} $.
\end{theorem}
\begin{proof}
The choice of the sub-solution is exactly the same as in Theorem \ref{negative:thm1}. When we fix the sub-solution $ u_{-} $, we choose $ u_{+} = C \gg 1 $ with $ C \geqslant u_{-} $, also large enough so that the same argument in Theorem \ref{negative:thm1} holds. Fix this $ C $ from now on. The only difference is that since $ H_{2} > 0 $ somewhere, we may need to shrink $ c $, if necessary, so that
\begin{equation*}
\frac{\partial C}{\partial \nu} + \frac{2}{p-2} h_{g} C \geqslant \frac{2}{p-2} \cdot \sup_{\partial M} (cH_{2}) C^{\frac{p}{2}}
\end{equation*}
The rest of the argument is exactly the same as in Theorem \ref{negative:thm1}.
\end{proof}
% Remark for the limit of this method in the category of $ \eta_{1} < 0 $.
\begin{remark}\label{negative:re1}
The method of monotone iteration scheme has its limits, as we cannot obtain the prescribed mean curvature to be $ H $, due to the technical issue, see \cite[Thm.~2.4]{XU8}.
\end{remark}
\medskip

% The case $ S > 0 $ somewhere, $ H > 0 $ somewhere and $ \eta_{1} < 0 $. 
We now discuss the Case (iii). The following argument is inspired by Kazdan and Warner \cite{KW}. When $ \eta_{1} < 0 $, Kazdan and Warner showed that the key is to get the super-solution of (\ref{negative:eqn1}), if we are not using the variational method but instead the monotone iteration scheme. Next result shows that a super-solution of (\ref{negative:eqn1}) can be converted to another relation. We point out that the following result is not specific for $ \eta_{1} < 0 $ case only.
\begin{lemma}\label{negative:lemma1}
Let $ (\bar{M}, g) $ be a compact manifold with non-empty smooth boundary $ \partial M $, $ n = \dim \bar{M} \geqslant 3 $. Let $ S, H \in \calC^{\infty}(\bar{M}) $ be given functions. Then there exists some positive function $ u \in \calC^{\infty}(\bar{M}) $ satisfying
\begin{equation}\label{negative:eqn4}
-a\Delta_{g} u + R_{g} u \geqslant S u^{p-1} \; {\rm in} \; M, \frac{\partial u}{\partial \nu} + \frac{2}{p - 2} h_{g} u \geqslant \frac{2}{p -2 } H u^{\frac{p}{2}} \; {\rm on} \; \partial M
\end{equation}
if and only if there exists some positive function $ w \in \calC^{\infty}(\bar{M}) $ satisfying
\begin{equation}\label{negative:eqn5}
-a\Delta_{g} w + (2 - p) R_{g} w + \frac{(p - 1)a}{p - 2} \cdot \frac{\lvert \nabla_{g} w \rvert^{2}}{w} \leqslant (2 - p)S \; {\rm in} \; M, \frac{\partial w}{\partial \nu} - 2 h_{g} w \leqslant -2H w^{\frac{1}{2}}.
\end{equation}
Moreover, the equality in (\ref{negative:eqn4}) holds if and only if the equality in (\ref{negative:eqn5}) holds.
\end{lemma}
\begin{proof}
Assume that there is a positive function $ u \in \calC^{\infty}(M) $ that satisfies (\ref{negative:eqn4}). Define
\begin{equation*}
w = u^{2 - p}.
\end{equation*}
Note that $ 2 - p = -\frac{4}{n - 2} < 0 $ since $ n \geqslant 3 $ by hypothesis. We compute that
\begin{equation*}
\nabla w = (2 - p) u^{1 - p} \nabla u \Leftrightarrow \nabla u = u^{p-1} (2 - p)^{-1} \nabla w,
\end{equation*}
and 
\begin{equation*}
\Delta_{g} w = (2 - p) u^{1 - p} \Delta_{g} u + (2 - p)(1 - p) u^{-p} \lvert \nabla_{g} u \rvert^{2}.
\end{equation*}
By the inequality (\ref{negative:eqn4}), we have
\begin{align*}
a\Delta_{g} w & = (2 - p) u^{1 - p} \left( a\Delta_{g} u \right) +a (2 - p) (1 - p) u^{-p} \lvert \nabla_{g} u \rvert^{2} \\
& \geqslant (p - 2) u^{1-p} \left( - R_{g} u + S u^{p-1} \right)+ a(2 - p) (1 - p) (2 - p)^{-2} u^{2p - 2} u^{-p} \lvert \nabla_{g} v \rvert^{2} \\
& = (p - 2) S + (2 - p) R_{g} u^{1 - p} +  \frac{a(p - 1)}{p - 2} u^{p -2} \lvert \nabla_{g} v \rvert^{2} \\
& = (p - 2) S + (2 - p) R_{g} w + \frac{a(p - 1)}{p - 2} \frac{ \lvert \nabla_{g} w \rvert^{2}}{w}.
\end{align*}
Shifting $ (p -2) S $ to the left side and $ a\Delta_{g} w $ to the right side, we get the first part of the inequality (\ref{negative:eqn5}). For the boundary condition, recall that $ u = w^{\frac{1}{2 - p}} $ and $ p = \frac{2n}{n - 2} $, it follows that
\begin{align*}
& \frac{\partial u}{\partial \nu} + \frac{2}{p -2 } H u^{\frac{p}{2}} \geqslant  \frac{2}{p-2} h_{g} u \Leftrightarrow \frac{1}{2 - p} w^{\frac{1}{2 - p} - 1} \frac{\partial w}{\partial \nu} + \frac{2}{p - 2} h_{g} w^{\frac{1}{2 - p}} \geqslant \frac{2}{p - 2} H w^{\frac{p}{2(2 - p)}} \\
\Leftrightarrow & -\frac{n - 2}{4} w^{-\frac{n}{4} - \frac{1}{2}} \frac{\partial w}{\partial \nu} + \frac{n - 2}{2} h_{g} w^{-\frac{n}{4} + \frac{1}{2}} \geqslant \frac{n-2}{2} H w^{-\frac{n}{4}} \\
\Leftrightarrow & \frac{\partial w}{\partial \nu} - 2h_{g} w \leqslant -2H w^{\frac{1}{2}}.
\end{align*}
Hence the second part of (\ref{negative:eqn5}) holds. It is clear that the equality holds if an only if all inequalities above are equalities.

If we assume (\ref{negative:eqn5}) for some $ w $, we just define $ u = w^{\frac{1}{2 - p}} $. The argument is very similar and we omit the details.
\end{proof}
% Third main result here.
We now introduce the result of prescribing scalar and mean curvature functions for Case (iii), with a technical restriction very similar to the condition given by Kazdan and Warner. This technical condition, in principle, is to show the positivity of the function that satisfies (\ref{negative:eqn5}). Due to the Han-Li conjecture \cite[Theorem]{XU5}, we may assume that the initial metric $ g $ has $ R_{g} = \lambda < 0 $ and $ h_{g} = \zeta > 0 $, since $ \eta_{1} < 0 $. Before we start with the special case, recall that if there exists a constant $ q > n $, and some function $ u \in \calC^{\infty}(\bar{M}) $ satisfies
\begin{equation*}
\lVert u \rVert_{W^{2, q}(M, g)} \leqslant \gamma' \left( \lVert F_{1} \rVert_{\calL^{q}(M, g)} + \lVert F_{2} \rVert_{W^{1, q}(M, g)} \right)
\end{equation*}
for some functions $ F_{1} \in \calL^{q}(M, g) $ and $ F_{2} \in W^{1, q}(M, g) $, the H\"older estimates implies that
\begin{equation}\label{negative:eqns1}
\lVert u \rVert_{\calL^{\infty}(\bar{M})} + \lVert \nabla u \rVert_{\calL^{\infty}(\bar{M})} \leqslant \gamma \left( \lVert F_{1} \rVert_{\calL^{q}(M, g)} + \lVert F_{2} \rVert_{W^{1, q}(M, g)} \right).
\end{equation}
This inequality is due to the Sobolev embedding in \cite[\S2]{Aubin}.
\begin{theorem}\label{negative:thm3}
Let $ (\bar{M}, g) $ be a compact manifold with non-empty smooth boundary $ \partial M $, $ n = \dim \bar{M} \geqslant 3 $. Assume that $ \eta_{1} < 0 $, $ R_{g} = \lambda < 0 $ and $ h_{g} = \zeta > 0 $ for some constants $ \lambda $ and $ \zeta $. Let $ S_{3}, H_{3} \in \calC^{\infty}(\bar{M}) $ and $ q > n $ be a positive integer. Let $ \gamma $ be the constant in the estimate (\ref{negative:eqns1}). Set $ D = \frac{(p - 1)a}{ p - 2} $. If there exists a function $ F \in \calC^{\infty}(\bar{M}) $ and a positive constant $ A > 0 $, such that
\begin{equation}\label{negative:eqn6}
( 2 - p ) S_{3} \geqslant F \; {\rm on} \; \partial M, \lVert F - A \rVert_{\calL^{q}(M, g)} \leqslant \frac{A}{2\gamma \left( 1 + \left(D + 1 \right) (2 - p) \lambda \right)},
\end{equation}
then there exists a small enough constant $ c > 0 $ such that (\ref{negative:eqn1}) admits a positive solution $ u \in \calC^{\infty}(\bar{M}) $ with $ S = S_{3} $ and $ H = cH_{3} $. Equivalently, there exists a Yamabe metric $ \tilde{g} = u^{p-2} g $ such that $ R_{\tilde{g}} = S_{3} $ and $ h_{\tilde{g}} = cH_{3} \bigg|_{\partial M} $.
\end{theorem}
\begin{proof} In this proof, we always denote $ R_{g} = \lambda $ and $ h_{g} = \zeta $. We construct the super-solution of (\ref{negative:eqn1}) first. Due to Lemma \ref{negative:lemma1}, it is equivalent to show the existence of some positive function $ w \in \calC^{\infty}(\bar{M}) $ such that (\ref{negative:eqn5}) holds for $ S = S_{3} $ and $ H = cH_{3} $ for some constant $ c $. Take
\begin{equation*}
\delta : =  \frac{A}{\left( 1 + ( D + 1 )(2 - p) \lambda \right)} > 0 .
\end{equation*}
We also choose some negative constant
\begin{equation*}
\delta' = - \frac{\delta}{2 \gamma \text{Vol}_{g}(M)^{\frac{1}{q}} } < 0.
\end{equation*}
By standard elliptic theory, there exists a unique solution $ w $ of the following PDE
\begin{equation*}
-a\Delta_{g} w + (2 - p) \lambda w = F - \delta \; {\rm in} \; M, \frac{\partial w}{\partial \nu} = \delta' \; {\rm on} \; \partial M.
\end{equation*}
The uniqueness comes from the fact that $ (2 - p) \lambda > 0 $, which implies the invertibility of the operator $ \left(-a\Delta_{g} + (2 -p) \lambda, \frac{\partial}{\partial \nu} \right) $. Clearly the constant $ (D + 1)\delta $ solves the PDE
\begin{equation*}
-a\Delta_{g} (( D + 1)\delta) + (2 - p) \lambda \cdot (( D + 1)\delta) = (2 - p) \lambda \cdot (( D + 1)\delta) \; {\rm in} \; \partial M, \frac{\partial (( D + 1)\delta)}{\partial \nu} = 0 \; {\rm on} \; \partial M.
\end{equation*}
Denote
\begin{equation*}
w_{0} : = w - ( D + 1)\delta.
\end{equation*}
The function $ w_{0} $ satisfies
\begin{equation}\label{negative:eqn7}
\begin{split}
-a\Delta_{g} w_{0} + (2 - p) \lambda w_{0} & = F - \delta - (D + 1)(2 - p) \lambda \delta = F - A \; {\rm in} \; M, \\
\frac{\partial w}{\partial \nu} & = \delta' \; {\rm on} \; \partial M.
\end{split}
\end{equation}
The first line in (\ref{negative:eqn7}) is due to the definition of $ \delta $. Since the differential operator with the boundary operator is invertible, we apply $ W^{s, q} $-type elliptic estimates (\ref{HL:eqn5}) as well as the estimates of (\ref{negative:eqns1}),
\begin{equation}\label{negative:eqn8}
\begin{split}
\lVert w_{0} \rVert_{\calL^{\infty}(\bar{M})} + \lVert \nabla w_{0} \rVert_{\calL^{\infty}(\bar{M})} & \leqslant \gamma \left( \lVert F - A \rVert_{\calL^{q}(M, g)} + \lVert \delta' \rVert_{W^{1, q}(M, g)} \right) \\
& \leqslant \gamma \left( \frac{A}{2\gamma \left( 1 + (D + 1)(2 - p) \lambda \right)} + \lvert \delta' \rvert \cdot \text{Vol}_{g}(M)^{\frac{1}{q}} \right) \\
& \leqslant \delta.
\end{split}
\end{equation}
The last inequality is due to the definitions of $ \delta $ and $ \delta' $. By definition of $ w_{0} $, the inequality (\ref{negative:eqn8}) implies
\begin{equation*}
\left\lVert w - (D + 1) \delta \right\rVert_{\calL^{\infty}(\bar{M})} \leqslant \delta, \lVert \nabla w \rVert_{\calL^{\infty}(\bar{\Omega})} \leqslant \delta.
\end{equation*}
It follows that
\begin{equation}\label{negative:eqn9}
0 < D \delta \leqslant w \leqslant (D + 2) w \; {\rm on} \; \bar{M}, \sup_{\bar{M}} \lvert \nabla w \rvert \leqslant \delta \Rightarrow \frac{(p - 1)a}{p - 2} \cdot \frac{\lvert \nabla w \rvert^{2}}{w} \leqslant \delta \; {\rm on} \; \bar{M}.
\end{equation}
With (\ref{negative:eqn6}), (\ref{negative:eqn9}), we have
\begin{align*}
& -a\Delta_{g} w + (2 - p) \lambda w + \frac{( p - 1)a}{ p -2} \cdot \frac{\lvert \nabla w \rvert^{2}}{w} = F - \delta + \frac{( p - 1)a}{ p -2} \cdot \frac{\lvert \nabla w \rvert^{2}}{w} \\
\leqslant & (2 - p) S_{3}; \\
& \frac{\partial w}{\partial \nu} - 2h_{g} w = \delta'  - 2h_{g} w \leqslant -2 \cdot \left( cH_{3} \right) w^{\frac{1}{2}}.
\end{align*}
The last inequality holds for small enough constant $ c > 0 $, regardless of the sign of $ H_{3} $ since $ \delta' - 2h_{g} w < 0 $ by set-up. By (\ref{negative:eqn9}) again, we conclude that $ w > 0 $ on $ \bar{M} $. By Lemma \ref{negative:lemma1}, the positive, smooth function
\begin{equation*}
u = w^{\frac{1}{2 - p}}
\end{equation*}
is a super-solution of (\ref{negative:eqn1}) with $ S = S_{3} $ and $ H = cH_{3} $. Note that $ u $ is still a super-solution if we make $ c $ smaller.

For sub-solution, we apply the perturbed eigenvalue problem in Proposition \ref{HL:prop3} again. There exists a small enough constant $ \beta > 0 $ such that
\begin{equation}\label{negative:eqn10}
-a\Delta_{g} \varphi + \lambda \varphi = \eta_{1, \beta} \varphi \; {\rm in} \; M, \frac{\partial \varphi}{\partial \nu} + \frac{2}{p-2} \left( \zeta + \beta \right) \varphi = 0 \; {\rm on} \; \partial M.
\end{equation}
Any scaling of $ \varphi $ solves (\ref{negative:eqn10}). Set the positive constant $ \xi \ll 1 $ such that
\begin{equation*}
\phi : = \xi \varphi \leqslant u \; {\rm on} \; \bar{M}
\end{equation*}
for the fixed super-solution $ u $ defined just above. We shrink $ \xi $ further, if necessary, such that
\begin{equation}\label{negative:eqn10}
\eta_{1, \beta} \left( \xi \varphi \right) \leqslant S_{3} \left( \xi \varphi \right)^{p-1} \; {\rm in} \; M, -\frac{2}{p-2} \cdot \beta \left( \xi \varphi \right) \leqslant \frac{2}{p-2} \cdot \left( cH_{3} \right) \cdot \left( \xi \varphi \right)^{\frac{p}{2}} \; {\rm on} \; \partial M.
\end{equation}
Note that the boundary condition holds for every $ c $, as long as we take $ \xi $ small enough. We point out that the choice of the constant $ c $ depends on the construction of the super-solution as well as the technical condition of the monotone iteration scheme, which only depends on the super-solution but not the sub-solution, see Equation (19) in \cite{XU8}. Thus we can choose $ c $ first, then determine $ \xi $. It follows that $ \phi $ is a sub-solution of (\ref{negative:eqn1}) with $ S = S_{3} $ and $ H = cH_{3} $. Furthermore, $ 0 < \phi \leqslant u $ on $ \bar{M} $. Applying Theorem \ref{HL:thm2}, we conclude that there exists a positive function $ u \in \calC^{\infty}(\bar{M}) $ as desired.
\end{proof}
% General prescribing scalar and mean curvatures when $ \eta_{1} < 0 $.
The general case when $ \eta_{1} < 0 $ is a straightforward consequence of the result above.
\begin{corollary}\label{negative:cor1}
Let $ (\bar{M}, g) $ be a compact manifold with non-empty smooth boundary $ \partial M $, $ n = \dim \bar{M} \geqslant 3 $. Let $ S_{4}, H_{4} \in \calC^{\infty}(\bar{M}) $ and $ q > n $ be a positive integer. Let $ \gamma $ be the constant in the estimate (\ref{negative:eqns1}) and $ \lambda $ be some negative constant. Set $ D = \frac{(p - 1)a}{ p - 2} $. Assume that $ \eta_{1} < 0 $. If there exists a function $ F \in \calC^{\infty}(\bar{M}) $ and a positive constant $ A > 0 $, such that
\begin{equation}\label{negative:eqn6}
( 2 - p ) S_{4} \geqslant F \; {\rm on} \; \partial M, \lVert F - A \rVert_{\calL^{q}(M, g)} \leqslant \frac{A}{2\gamma \left( 1 + \left(D + 1 \right) (2 - p) \lambda \right)},
\end{equation}
then there exists a small enough constant $ c > 0 $ such that (\ref{negative:eqn1}) admits a positive solution $ u \in \calC^{\infty}(\bar{M}) $ with $ S = S_{4} $ and $ H = cH_{4} $. Equivalently, there exists a Yamabe metric $ \tilde{g} = u^{p-2} g $ such that $ R_{\tilde{g}} = S_{4} $ and $ h_{\tilde{g}} = cH_{4} \bigg|_{\partial M} $.
\end{corollary}
\begin{proof}
By the result of the Han-Li conjecture \cite[Theorem]{XU5}, there exists a conformal metric $ g_{1} = v^{p-2} g $ such that $ R_{g_{1}} = \lambda $ and $ h_{g_{1}} = \zeta $. We then apply Theorem \ref{negative:thm3} for the metric $ g_{1} $, i.e. there exists $ \tilde{g} = u^{p-2} g_{1} $ with $ R_{\tilde{g}} = S_{4} $ and $ h_{\tilde{g}} = cH_{4} $ with small enough $ c > 0 $. The conformal change
\begin{equation*}
\tilde{g}= \left( uv \right)^{p-2} g
\end{equation*}
is the desired metric.
\end{proof}
\medskip

% This is the section for prescribed scalar curvature and mean curvature when $ \eta_{1} < 0 $.
\section{Prescribed Scalar and Mean Curvature Functions for Conformal Equivalent Metrics When $ \eta_{1} < 0 $}
Inspired by the ``Trichotomy Theorem" on closed manifolds, we would like to discuss the prescribing scalar and mean curvature problem on $ (\bar{M}, g) $, $ n = \dim \bar{M} \geqslant 3 $, but not restricted in a conformal class $ [g] $ only. Instead, we are interested in the conformally equivalent metrics. 
\begin{definition}\label{ne:def1}
Let $ (\bar{M}, g) $ be a compact manifold with non-empty smooth boundary $ \partial M $, we say that a metric $ \tilde{g} $ is conformally equivalent to the metric $ g $ if there exists a positive, smooth function $ u \in \calC^{\infty}(\bar{M}) $ and a diffeomorphism $ \phi : \bar{M} \rightarrow \bar{M} $ such that
\begin{equation*}
\phi^{*} \tilde{g} = u^{p-2} g.
\end{equation*}
\end{definition}
Within in a conformal class, the prescribing scalar and mean curvature problem for given functions $ S, H \in \calC^{\infty}(\bar{M}) $ is reduced to the PDE (\ref{intro:eqn1}). For conformlaly equivalent metrics, the prescribing scalar and mean curvature problem is reduced to the existence of a positive, smooth solution of the following PDE
\begin{equation}\label{ne:eqn1}
-a\Delta_{g} u + R_{g} u = \left( S \circ \phi \right) u^{p-1} \; {\rm in} \; M, \frac{\partial u}{\partial \nu} + \frac{2}{p-2} h_{g} u = \frac{2}{p-2} \cdot \left(H \circ \phi \right) \cdot u^{\frac{p}{2}} \; {\rm on} \; \partial M.
\end{equation}
Our next result extends the result of prescribing scalar curvature problem on closed manifolds with dimensions at least $ 3 $ \cite[Thm.~3.3]{KW} to compact manifolds with non-empty smooth boundaries, provided that the first eigenvalue $ \eta_{1} $ of the conformal Laplacian with Robin boundary condition is negative. The method is essentially due to Kazdan and Warner \cite{KW2, KW}.
% Conformally equivalent case, prescribing scalar and mean curvatures, $ \eta_{1} < 0 $.
\begin{theorem}\label{ne:thm1}
Let $ (\bar{M}, g) $ be a compact manifold with non-empty smooth boundary $ \partial M $, $ n = \dim \bar{M} \geqslant 3 $. Let $ S_{5} $ be any smooth function on $ \bar{M} $ that is negative somewhere in $ M $. Let $ H_{5} \in \calC^{\infty}(\bar{M}) $ and $ q > n $ be a positive integer. If $ \eta_{1} < 0 $, then there exists a small enough constant $ c > 0 $ and a diffeomorphism $ \phi : \bar{M} \rightarrow \bar{M} $ such that (\ref{ne:eqn1}) admits a positive solution $ u \in \calC^{\infty}(\bar{M}) $ with $ S = S_{5} $ and $ H = c  H_{5} $. Equivalently, there exists a conformally equivalent metric $ \tilde{g} = \left( \phi^{-1} \right)^{*} \left( u^{p-2} g \right) $ such that $ R_{\tilde{g}} = S_{5} $ and $ h_{\tilde{g}} = cH_{5} \bigg|_{\partial M} $.
\end{theorem}
\begin{proof} By Han-Li conjecture \cite[Theorem]{XU5}, we may assume that $ R_{g} = \lambda < 0 $ and $ h_{g} = \zeta > 0 $ for some constants $ \lambda, \zeta $. Fix some constant $ q > n $. Due to Theorem \ref{negative:thm3}, it suffices to find a diffeomorphism $ \phi : \bar{M} \rightarrow \bar{M} $, a smooth function $ F \in \calC^{\infty}(\bar{M}) $, a positive constant $ A > 0 $ and a small enough positive constant $ c $ such that
\begin{equation}\label{ne:eqn2}
( 2 - p) S_{5} \circ \phi \geqslant F \; {\rm on} \; \bar{M}, \lVert F - A \rVert_{\calL^{q}(M, g)} \leqslant \frac{A}{2 \gamma \left( 1 + (D + 1)(2 - p)\lambda \right) };
\end{equation}
in addition, $ \sup_{\bar{M}} c \left( H \circ \phi \right) $ is small enough. Here $ \gamma $ is the constant in the estimate (\ref{negative:eqns1}), the constant $ D $ is defined to be
\begin{equation*}
D = \frac{( p -1)a}{p - 2}.
\end{equation*}
We determine $ \phi $, $ F $ and $ A $ first. If $ S_{5} < 0 $ everywhere on $ \bar{M} $, we just choose $ \phi $ to be the identity map and set
\begin{equation*}
F = A = (2 - p) \max_{\bar{M}} S_{5}.
\end{equation*}
It is straightforward to check that (\ref{ne:eqn2}) holds.

If $ S_{5} \geqslant 0 $ somewhere and changes sign, we choose $ A $ first to be any positive constant such that
\begin{equation}\label{ne:eqn3}
0 < A  < (2 - p) \min_{\bar{M}} S_{5}.
\end{equation}
Just note that $ (2 - p) < 0 $. We pick interior open submanifolds $ U, V \subset M $ such that
\begin{equation*}
V \subset \bar{V} \subset U \subset M \subset \bar{M}.
\end{equation*}
In particular, we require that
\begin{equation}\label{ne:eqn4}
\text{Vol}_{g}(U - V) \leqslant  \left( \frac{A}{2 \gamma \left( 1 + (D + 1) ( 2- p ) \lambda \right) \cdot  \left( ( 2 - p) \lVert S_{3} \rVert_{\calL^{\infty}(\bar{M})} - A \right) } \right)^{q}.
\end{equation}
We select the diffeomorphism $ \phi $ such that
\begin{equation}\label{ne:eqn5}
( 2 - p) S_{3} \circ \phi > A \; {\rm in} \; U.
\end{equation}
We then take the function $ F $ to be
\begin{equation}\label{ne:eqn5a}
\begin{split}
& F = A \; {\rm in} \; V; \\
& ( 2- p ) \max_{\bar{M}} S_{3} \circ \phi \leqslant F \leqslant A \; {\rm in} \; U - V; \\
& F = ( 2- p ) \max_{\bar{M}} S_{3} \circ \phi \; {\rm in} \; \bar{M} - U.
\end{split}
\end{equation}
Clearly $ F \leqslant ( 2 - p) S_{3} \circ \phi $ on $ \bar{M} $ by (\ref{ne:eqn5a}). The function $ F $ only differs with $ A $ in $ U - V $, by (\ref{ne:eqn4}), it is immediate to check that the second inequality in (\ref{ne:eqn2}) holds. 

Lastly we choose $ c $ so that the condition in Theorem \ref{HL:thm2} holds for the function $ S_{3} \circ \phi $, i.e. $ c \sup_{\bar{M}} \lvert H_{5} \rvert $ is small enough. The same $ c $ applies for the smallness of $ c \sup_{\bar{M}} \lvert H_{5} \circ \phi \rvert $ since the diffeomorphism does not change the extremal values of a function. Therefore the function $ S_{3} \circ \phi $ and $ cH_{5} \circ \phi $ can be realized as prescribed scalar and mean curvature functions, respectively, for some metric $ \phi^{*} \tilde{g} = u^{p-2} g $ where $ u $ is positive and smooth on $ \bar{M} $. Equivalently, $ S_{5} $ and $ cH_{5} $ can be realized as prescribed scalar and mean curvature functions, respectively, for some metric $ \tilde{g} = \left( \phi^{-1} \right)^{*} u^{p-2} g $.
\end{proof}
\begin{remark}\label{ne:re1}
The result of Theorem \ref{ne:thm1} indicates that on $ (\bar{M}, g) $ with $ n = \dim \bar{M} \geqslant 3 $, any function that is negative somewhere can be realized as a scalar curvature function of some metric $ g $, meanwhile the mean curvature function of $ g $ can be some small enough scaling of any smooth function, provided that the manifold admits a metric with negative first eigenvalue of the conformal Laplacian, or equivalently, negative Yamabe invariant \cite[\S1]{ESC}.
\end{remark}
\medskip

% This is the section for prescribing scalar curvature and mean curvature when $ \chi(M) < 0 $. Dimension 2.
\section{Prescribed Gauss and Geodesic Curvature Functions When $ \chi(\bar{M}) < 0 $}
In this section, we discuss the prescribing Gauss and geodesic curvatures problem within a conformal class $ [g] $ of compact manifolds $ (\bar{M}, g) $ with non-empty smooth boundary $ \partial M $, provided that $ \chi(\bar{M}) < 0 $ and $ n = \dim \bar{M} = 2 $. This is a $ 2 $-dimensional analogy of prescribing scalar and mean curvatures problem with $ \eta_{1} < 0 $, provided that the dimension is at least $ 3 $.

Let $ K, \sigma \in \calC^{\infty}(\bar{M}) $ be given functions. This type of Kazdan-Warner problem is reduced to the existence of a smooth solution $ u $ of the following PDE
\begin{equation}\label{ne2:eqn1}
-a\Delta_{g} u + K_{g} = K e^{2u} \; {\rm in} \; M, \frac{\partial u}{\partial \nu} + \sigma_{g} = \sigma e^{u} \; {\rm on} \; \partial M.
\end{equation}
Here $ K_{g} $ and $ \sigma_{g} $ are Gaussian and geodesic curvatures of $ g $, respectively. The solvability of this PDE implies that the metric $ \tilde{g} = e^{2u} g $ has Gauss curvature $ K_{\tilde{g}} = K $ and geodesic curvature $ \sigma_{\tilde{g}} = \sigma $. We mainly discuss to cases:
\begin{enumerate}[(i).]
\item $ K \leqslant 0 $ everywhere in $ \bar{M} $, and arbitrary $ \sigma $, with $ \chi(\bar{M}) < 0 $;
\item $ K > 0 $ somewhere in $ \bar{M} $ and changes sign, $ \sigma $ is an arbitrary function, with $ \chi(\bar{M}) < 0 $.
\end{enumerate}
We would like to apply the monotone iteration scheme to solve (\ref{ne2:eqn1}), it is equivalent to construct the sub- and super-solutions of (\ref{ne2:eqn1}). The key is to construct the super-solution. As in \S3, we convert the super-solution of (\ref{ne2:eqn1}) into another inequality involving derivatives. 
% One PDE is equivalent to another PDE, 2-d version.
\begin{lemma}\label{ne2:lemma1}
Let $ (\bar{M}, g) $ be a compact manifold with non-empty smooth boundary $ \partial M $, $ n = \dim \bar{M} = 2 $. Let $ K, \sigma \in \calC^{\infty}(\bar{M}) $ be given functions. Then there exists some function $ u \in \calC^{\infty}(\bar{M}) $ satisfying
\begin{equation}\label{ne2:eqn2}
-\Delta_{g} u + K_{g} \geqslant K e^{2u} \; {\rm in} \; M, \frac{\partial u}{\partial \nu} + \sigma_{g} \geqslant  \sigma e^{u} \; {\rm on} \; \partial M
\end{equation}
if and only if there exists some positive function $ w \in \calC^{\infty}(\bar{M}) $ satisfying
\begin{equation}\label{ne2:eqn3}
-\Delta_{g} w - 2wK_{g} + \frac{\lvert \nabla_{g} w \rvert^{2}}{w} \leqslant -2K \; {\rm in} \; M, \frac{\partial w}{\partial \nu} - 2w \sigma_{g} \leqslant -2\sigma w^{\frac{1}{2}} \; {\rm on} \; \partial M.
\end{equation}
Moreover, the equality in (\ref{ne2:eqn2}) holds if and only if the equality in (\ref{ne2:eqn3}) holds; and the inequality in (\ref{ne2:eqn2}) is in the reverse direction if and only if the inequality in (\ref{ne2:eqn3}) is in the reverse direction.
\end{lemma}
\begin{proof} Assume (\ref{ne2:eqn2}) for some function $ u $ first. Define
\begin{equation*}
w : = e^{-2u}
\end{equation*}
We observe that
\begin{align*}
\nabla_{g} w & = -2 e^{-2u} \nabla_{g} u \Rightarrow \nabla_{g} u = -\frac{1}{2} e^{2u} \nabla_{g} w, \\
\Delta_{g} w & = -2e^{-2u} \Delta_{g} u + 4e^{-2u} \lvert \nabla_{g} u \rvert^{2} = -2e^{-2u} \Delta_{g} u + e^{2u} \lvert \nabla_{g} w \rvert^{2}.
\end{align*}
Thus we have
\begin{align*}
-\Delta_{g} w & = 2 e^{-2u} \Delta_{g} u - e^{2u} \lvert \nabla_{g} w \rvert^{2} \leqslant 2e^{-2u} \left( K_{g} - K e^{2u} \right) - \frac{\lvert \nabla_{g} w \rvert^{2}}{w} \\
& = 2wK_{g} - 2K - \frac{\lvert \nabla_{g} w \rvert^{2}}{w} \\
\Rightarrow & -\Delta_{g} w - 2wK_{g} + \frac{\lvert \nabla_{g} w \rvert^{2}}{w} \leqslant -2K \; {\rm in} \; M.
\end{align*}
For the boundary condition, we have
\begin{align*}
\frac{\partial w}{\partial \nu} & = \frac{\partial e^{-2u}}{\partial \nu} = -2e^{-2u} \frac{\partial u}{\partial \nu} \leqslant -2e^{-2u} \left( -\sigma_{g} + \sigma e^{u} \right) \\
& = 2w\sigma_{g} - 2\sigma w^{\frac{1}{2}} \\
\Rightarrow & \frac{\partial w}{\partial \nu} - 2w \sigma_{g} \leqslant -2\sigma w^{\frac{1}{2}} \; {\rm on} \; \partial M.
\end{align*}
Therefore (\ref{ne2:eqn3}) holds for $ w = e^{-2u} > 0 $ on $ \bar{M} $. It is clear that equality holds when all inequalities above are equalities. It is also straightforward to see that the inequalities are in the reverse directions if and only if the inequalities are in the reverse directions in each step above.

For the opposite direction, we assume (\ref{ne2:eqn3}) holds for some positive, smooth function $ w $. Define
\begin{equation*}
u = -\frac{1}{2} \log w.
\end{equation*}
We can show that $ u $ satisfies (\ref{ne2:eqn2}). The argument is quite similar to above and we omit the details.
\end{proof}
Due to the uniformization theorem, we may assume $ K_{g} = -1 $ and $ \sigma_{g} = 0 $ in (\ref{ne2:eqn1}) from now on, as our model case up to some pointwise conformal change, provided that $ \chi(\bar{M}) < 0 $. In $ 2 $-dimensional case, we also have the $ W^{s, q} $-type estimates from Theorem \ref{HL:thm1}. We choose $ q = 3 $, the estimate in (\ref{HL:eqn5}) plus the Sobolev embedding into H\"older space, the inequality in (\ref{negative:eqns1}) becomes
\begin{equation}\label{ne2:eqn4}
\lVert u \rVert_{\calL^{\infty}(\bar{M})} + \lVert \nabla u \rVert_{\calL^{\infty}(\bar{M})} \leqslant \gamma \left( \lVert F_{1} \rVert_{\calL^{3}(M, g)} + \lVert F_{2} \rVert_{W^{1, 3}(M, g)} \right).
\end{equation}
Here $ F_{1}, F_{2} $ and $ u $ comes from the PDE (\ref{HL:eqn4}) with the operators $ L = - \Delta_{g} + 2 $ and $ B = \frac{\partial}{\partial \nu} $, so is the constant $ \gamma $. Our main result of this section is the following, which covers both Case (i) and Case (ii) at the beginning of this section.
% The major result of this section, covere both cases, all rest results are corollaries.
\begin{theorem}\label{ne2:thm1}
Let $ (\bar{M}, g) $ be a compact Riemann surface with non-empty smooth boundary $ \partial M $. Let $ K_{1}, \sigma_{1} \in \calC^{\infty}(\bar{M}) $ be given functions. Let $ \gamma $ be the constant in the estimate (\ref{ne2:eqn4}). Assume that $ \chi(\bar{M}) < 0 $. If there exists a function $ F \in \calC^{\infty}(\bar{M}) $ and a positive constant $ A > 0 $, such that
\begin{equation}\label{ne2:eqn5}
-2K_{1} \geqslant F \; {\rm on} \; \partial M, \lVert F - A \rVert_{\calL^{3}(M, g)} \leqslant \frac{A}{6\gamma},
\end{equation}
then there exists a small enough constant $ c > 0 $ such that (\ref{ne2:eqn1}) admits a positive solution $ u \in \calC^{\infty}(\bar{M}) $ with $ K = K_{1} $ and $ \sigma = c\sigma_{1} $. Equivalently, there exists a Yamabe metric $ \tilde{g} = e^{2u} g $ such that $ K_{\tilde{g}} = K_{1} $ and $ \sigma_{\tilde{g}} = c\sigma_{1} \bigg|_{\partial M} $.
\end{theorem}
\begin{proof}
The proof is essentially the same as in Theorem \ref{ne:thm1}. By Lemma \ref{ne2:eqn1}, the construction of the super-solution is equivalent to the construction of a function $ w $ that satisfies (\ref{ne2:eqn3}) for $ K_{1}, \sigma_{1} $ and some small enough positive constant $ c $. We set
\begin{equation}\label{ne2:eqn6}
\delta =  \frac{A}{3}, \delta' = -\frac{\delta}{2 \gamma \text{Vol}_{g}(M)^{\frac{1}{3}}}.
\end{equation}
There is a unique solution for the PDE
\begin{equation*}
-\Delta_{g} w + 2w = F - \delta \; {\rm in} \; M, \frac{\partial w}{\partial \nu} = \delta' \; {\rm on} \; \partial M.
\end{equation*}
Define
\begin{equation*}
w_{0} = w - 2\delta,
\end{equation*}
it follows that $ w_{0} $ satisfies the PDE
\begin{equation}\label{ne2:eqn7}
-\Delta_{g} w_{0} + 2w_{0} = F - 3\delta = F - A \; {\rm in} \; M, \frac{\partial w_{0}}{\partial \nu} = \delta' \; {\rm on} \; \partial M.
\end{equation}
Apply the estimate (\ref{ne2:eqn4}) for $ w_{0} $ in (\ref{ne2:eqn7}), it follows that
\begin{equation*}
\lVert w_{0} \rVert_{\calL^{\infty}(\bar{M})} + \lVert \nabla w_{0} \rVert_{\calL^{\infty}(\bar{M})} \leqslant \gamma \left( \lVert F - A \rVert_{\calL^{3}(M, g)} + \lVert \delta' \rVert_{W^{1, 3}(M, g)} \right) \leqslant \delta.
\end{equation*}
It follows from the definition of $ w_{0} $ that
\begin{equation*}
0 < \delta \leqslant w \leqslant 3\delta \; {\rm on} \; \bar{M}, \lVert \nabla w \rVert_{\calL^{\infty}(\bar{M})} \leqslant \delta.
\end{equation*}
Therefore we conclude that
\begin{equation*}
-\Delta_{g} w + 2w + \frac{\lvert \nabla w \rvert^{2}}{w} = F - \delta + \frac{\lvert \nabla w \rvert^{2}}{w} \leqslant F \leqslant -2K_{1} \; {\rm in} \; M.
\end{equation*}
In addition, we take $ c $ small enough so that
\begin{equation*}
\frac{\partial w}{\partial \nu} = \delta' \leqslant -2c\sigma_{1} w^{\frac{1}{2}} \; {\rm on} \; \partial M.
\end{equation*}
This can be done since $ \delta' < 0 $. It follows that the function
\begin{equation*}
u_{+} : = -\frac{1}{2} \log w 
\end{equation*}
is a super-solution of (\ref{ne2:eqn1}) with $ K = K_{1} $ and $ \sigma = c\sigma_{1} $. Clearly $ u_{+} \in \calC^{\infty}(\bar{M}) $. 

We construct a sub-solution now. Consider the PDE
\begin{equation*}
-\Delta_{g} u_{0} = \frac{1}{2} \; {\rm in} \; M, \frac{\partial u_{0}}{\partial \nu} = C \; {\rm on} \; \partial M.
\end{equation*}
By standard elliptic PDE theory, see e.g. \cite[Prop.~7.7, Ch.~4]{T}, the above PDE is solvable by some smooth function $ u_{0} \in \calC^{\infty}(\bar{M}) $ if $ -\int_{M} \frac{1}{2} \dvol = \int_{\partial M} C dS_{g} $. We choose the constant $ C < 0 $ so that the compatibility condition just mentioned holds. Clearly 
\begin{equation*}
u_{-} : = u_{0} + C_{1}
\end{equation*}
solves the PDE above also for any constant $ C_{1} $. We just choose $ C_{1} $ to be very negative such that
\begin{equation*}
u_{-} \leqslant u_{+} \; {\rm on} \; \partial M;
\end{equation*}
In addition,
\begin{align*}
-\Delta_{g} u_{-} - 1 & = -\frac{1}{2} \leqslant K_{1} e^{2u_{-}} = K_{1} e^{2u_{0}} \cdot e^{2C_{1}} \; {\rm in} \; M, \\
\frac{\partial u_{-}}{\partial \nu} & = C \leqslant c\sigma_{1} e^{u_{-}} = c\sigma_{1} e^{u_{0}} \cdot e^{C_{1}} \; {\rm on} \; \partial M.
\end{align*}
These can be done since the constants on the left sides of the inequalities are both negative. Note that both the super-solution and sub-solution holds for smaller constant $ c $ by adjusting the constant $ C_{1} $ only.

Note that when $ F(\cdot, u) = K_{1} e^{2u} + 1 $, $ G(\cdot, u) = c\sigma_{1} e^{u} $, the condition (\ref{HL:eqn19}) is independent of the sub-solution $ u_{-} $ as we can see the very similar case in Theorem \ref{HL:thm2} for the Yamabe equation. Thus we take $ c $ small enough so that the hypotheses in Theorem \ref{HL:thm3} holds. It follows that there exists some smooth function $ u $ that solves (\ref{ne2:eqn1}) with $ K = K_{1} $ and $ \sigma = c\sigma_{1} $.
\end{proof}
We can partially answer the two cases we are interested in. For Case (ii), not every function that changes sign can be a prescribed scalar curvature function unless it is not too positive too often. We show that every function that is negative everywhere can be realized as a scalar curvature function, meanwhile, a small enough scaling of any function can be realized as prescribed mean curvature function, under pointwise conformal deformation. This is Case (i).
% $ K < 0 $ everywhere and $ \sigma $ is arbitrary, with $ \chi(\bar{M}) < 0 $.
\begin{corollary}\label{ne2:cor1}
Let $ (\bar{M}, g) $ be a compact Riemann surface with non-empty smooth boundary $ \partial M $. Let $ K_{2}, \sigma_{2} \in \calC^{\infty}(\bar{M}) $ be given functions. Assume that $ K_{2} < 0 $ everywhere on $ \bar{M} $. If $ \chi(\bar{M}) < 0 $, then there exists a small enough constant $ c $, a smooth function $ u \in \calC^{\infty}(\bar{M}) $ such that $ u $ solves (\ref{ne2:eqn1}) with $ K = K_{2} $ and $ \sigma = c\sigma_{2} $. It is equivalent to say that the metric $ \tilde{g} = e^{2u} g $ has Gauss curvature $ K_{\tilde{g}} = K_{2} $ and geodesic curvature $ \sigma_{\tilde{g}} = c\sigma_{2} $.
\end{corollary}
\begin{proof}
We show that the condition (\ref{ne2:eqn5}) holds. Since $ K_{2} < 0 $ everywhere, we just choose
\begin{equation*}
F = A = -2\max_{\bar{M}} K_{2} \Rightarrow -2K_{2} \geqslant F, \lVert F - A \rVert_{\calL^{3}(M, g)} = 0.
\end{equation*}
We just need to choose a small enough $ c $ such that the hypotheses in Theorem \ref{ne2:thm1} and Theorem \ref{HL:thm3} hold.
\end{proof}
\medskip

For Case (ii), we can get a more comprehensive answer by considering the class of conformally equivalent metrics. Analogous to \S4, we are looking for a metric $ \tilde{g} = \left( \phi^{-1} \right)^{*} e^{2u} g $ with some diffeomorphism $ \phi : \bar{M} \rightarrow \bar{M} $ and smooth function $ u \in \calC^{\infty}(\bar{M}) $ such that the scalar and mean curvatures of $ \tilde{g} $ are given functions $ K, \sigma \in \calC^{\infty}(\bar{M}) $, respectively. This problem is reduced to the PDE
\begin{equation}\label{ne2:eqn8}
-\Delta_{g} u + K_{g} = \left(K \circ \phi \right)e^{2u} \; {\rm in} \; M, \frac{\partial u}{\partial \nu} + \sigma_{g} u = \left( \sigma \circ \phi \right) e^{u} \; {\rm on} \; \partial M.
\end{equation}
Similar to Theorem \ref{ne:thm1} for dimensions at least $ 3 $, we introduce the following result for compact Riemann surfaces.
% Conformally Equivalent Case, $ \chi(\bar{M}) < 0 $.
\begin{corollary}\label{ne2:cor2}
Let $ (\bar{M}, g) $ be a compact Riemann surface with non-empty smooth boundary $ \partial M $. Let $ \sigma_{3} \in \calC^{\infty}(\bar{M}) $ be any function and $ K_{3} \in \calC^{\infty}(\bar{M}) $ be a function that is negative somewhere in $ M $. If $ \chi(\bar{M}) < 0 $, then there exists a small enough constant $ c $, a smooth function $ u \in \calC^{\infty}(\bar{M}) $ and a diffeomorphism $ \phi : \bar{M} \rightarrow \bar{M} $ such that $ u $ solves (\ref{ne2:eqn8}) with $ K = K_{3} $ and $ \sigma = c\sigma_{3} $. It is equivalent to say that the metric $ \tilde{g} =\left( \phi^{-1} \right)^{*} e^{2u} g $ has Gauss curvature $ K_{\tilde{g}} = K_{3} $ and geodesic curvature $ \sigma_{\tilde{g}} = c\sigma_{3} $.
\end{corollary}
\begin{proof} The proof is essentially the same as in Theorem \ref{ne:thm1}. We determine $ \phi, F, A $ first so that (\ref{ne2:eqn5}) holds; then determine the constant $ c $. We may assume that $ K_{3} $ is negative somewhere but not everywhere since otherwise it is reduced to the result of Corollary \ref{ne2:cor1}.

We choose $ A $ first to be any positive constant such that
\begin{equation}\label{ne2:eqn9}
0 < A  < -2\min_{\bar{M}} K_{3}.
\end{equation}
We pick interior open submanifolds $ U, V \subset M $ such that
\begin{equation*}
V \subset \bar{V} \subset U \subset M \subset \bar{M}.
\end{equation*}
In particular, we require that
\begin{equation}\label{ne2:eqn10}
\text{Vol}_{g}(U - V) \leqslant  \left( \frac{A}{6 \gamma \cdot  \left( 2 \lVert K_{3} \rVert_{\calL^{\infty}(\bar{M})} - A \right) } \right)^{3}.
\end{equation}
We select the diffeomorphism $ \phi $ such that
\begin{equation}\label{ne2:eqn11}
-2 K_{3} \circ \phi > A \; {\rm in} \; U.
\end{equation}
We then take the function $ F $ to be
\begin{equation}\label{ne2:eqn12}
\begin{split}
& F = A \; {\rm in} \; V; \\
& -2 \max_{\bar{M}} K_{3} \circ \phi \leqslant F \leqslant A \; {\rm in} \; U - V; \\
& F = -2 \max_{\bar{M}} K_{3} \circ \phi \; {\rm in} \; \bar{M} - U.
\end{split}
\end{equation}
Clearly $ F \leqslant -2K_{3} \circ \phi $ on $ \bar{M} $ by (\ref{ne2:eqn12}). The function $ F $ only differs with $ A $ in $ U - V $, by (\ref{ne2:eqn10}), it is immediate to check that the second inequality in (\ref{ne2:eqn5}) holds. 

Lastly we choose $ c $ so that the condition in Theorem \ref{HL:thm3} holds for the function $ K_{3} \circ \phi $, i.e. $ c \sup_{\bar{M}} \lvert \sigma_{3} \rvert $ is small enough. The same $ c $ applies for the smallness of $ c \sup_{\bar{M}} \lvert \sigma_{3} \circ \phi \rvert $ since the diffeomorphism does not change the extremal values of a function. Therefore the function $ K_{3} \circ \phi $ and $ c\sigma_{3} \circ \phi $ can be realized as prescribed scalar and mean curvature functions, respectively, for some metric $ \phi^{*} \tilde{g} = u^{p-2} g $ where $ u $ is positive and smooth on $ \bar{M} $. Equivalently, $ K_{3} $ and $ c\sigma_{3} $ can be realized as prescribed scalar and mean curvature functions, respectively, for some metric $ \tilde{g} = \left( \phi^{-1} \right)^{*} u^{p-2} g $.
\end{proof}
\begin{remark}\label{ne2:re1}
The result of Corollary \ref{ne2:cor2}, combining Theorem \ref{ne:thm1} indicate that on $ (\bar{M}, g) $ with $ n = \dim \bar{M} \geqslant 2 $, any function that is negative somewhere can be realized as a scalar/Gauss curvature function of some metric $ g $, meanwhile the mean/geodesic curvature function of $ g $ can be some small enough scaling of any smooth function, provided that the manifold admits a metric with negative first eigenvalue of the conformal Laplacian, or negative Euler characteristics, respectively, depending on the dimension of the manifold. This improve the result mentioned in Remark \ref{ne:re1}.
\end{remark}
\medskip

% This is the section for prescribing scalar curvature and mean curvature when $ \eta_{1} = 0 $. Dimension at least 3.
\section{Prescribed Scalar and Mean Curvature Functions for Conformally Equivalent Metrics When $ \eta_{1} = 0 $}
In this section, we discuss the prescribing scalar and mean curvatures problem for metrics conformally equivalent to the metric $ g $ on compact manifolds $ (\bar{M}, g) $ with non-empty smooth boundary $ \partial M $, provided that $ \eta_{1} = 0 $ and $ n = \dim \bar{M} \geqslant 3 $. We gave a comprehensive study for manifolds with dimensions at least $ 3 $ in \cite{XU8} for pointwise conformal change. Here we consider whether there exists some smooth function $ u \in \calC^{\infty}(\bar{M}) $ and some diffeomorphism $ \phi : \bar{M} \rightarrow \bar{M} $ such that the metric $ \tilde{g} = \left( \phi^{-1} \right)^{*} u^{p-2} g $ has scalar curvature $ S $ and mean curvature $ H $ for some given functions $ S, H \in \calC^{\infty}(\bar{M}) $. Since the model case for zero first eigenvalue case is $ R_{g} = h_{g} = 0 $, the problem above is reduced to the existence of the solution of the following PDE
\begin{equation}\label{zero:eqn1}
-a\Delta_{g} u = \left( S \circ \phi \right) \cdot u^{p-1} \; {\rm in} \; M, \frac{\partial u}{\partial \nu} = \frac{2}{p-2} \cdot \left( H \circ \phi \right) \cdot u^{\frac{p}{2}} \; {\rm on} \; \partial M.
\end{equation}
Recall the result of prescribing scalar and mean curvature problems for conformal metrics on $ (\bar{M}, g) $.
\begin{theorem}\label{zero:thm1}\cite[Thm.~1.4]{XU8}
Let $ (\bar{M}, g) $ be a compact manifold with non-empty smooth boundary $ \partial M $, $ n = \dim \bar{M} \geqslant 3 $. Let $ S, H \in \calC^{\infty}(\bar{M}) $ be given nonzero functions. Assume that $ \eta_{1} = 0 $. If the function $ S $ satisfies
\begin{equation*}
\text{$ S $ changes sign and} \; \int_{M} S \dvol < 0,
\end{equation*}
then there exists a pointwise conformal metric $ \tilde{g} \in [g] $ that has scalar curvature $ R_{\tilde{g}} = S $ and $ h_{\tilde{g}} = cH $ for some small enough positive constant $ c $.
\end{theorem}
The conformally equivalent case follows from the result of Theorem \ref{zero:thm1}, we show it below. Note that the case $ S = H = 0 $ is the trivial case.
\begin{theorem}\label{zero:thm2}
Let $ (\bar{M}, g) $ be a compact manifold with non-empty smooth boundary $ \partial M $, $ n = \dim \bar{M} \geqslant 3 $. Let $ S_{6}, H_{6} \in \calC^{\infty}(\bar{M}) $ be given nonzero functions. Assume that $ \eta_{1} = 0 $. If the function $ S $ satisfies
\begin{equation*} 
\text{$ S_{6} $ changes sign},
\end{equation*}
then there exists a diffeomorphism $ \phi : \bar{M} \rightarrow \bar{M} $ and a small enough constant $ c > 0 $ such that (\ref{zero:eqn1}) has a smooth solution $ u \in \calC^{\infty}(\bar{M}) $ for $ \phi $, $ S = S_{6} $ and $ H =  cH_{6} $. It is equivalent to say that the conformally equivalent metric $ \tilde{g} = \left( \phi^{-1} \right)^{*} u^{p-2} g $ has scalar curvature $ R_{\tilde{g}} = S_{6} $ and mean curvature $ h_{\tilde{g}} = cH_{6} $.
\end{theorem}
\begin{proof} Due to Theorem \ref{zero:thm1}, it suffices to show that there exist a diffeomorphism $ \phi : \bar{M} \rightarrow \bar{M} $ such that
\begin{equation*}
\int_{M} \left( S_{6} \circ \phi \right) \dvol < 0.
\end{equation*}
Due to the same reason in \cite{KW3, KW2}, it is straightforward that such a diffeomorphism does exist since $ S_{6} $ changes sign. The smallness of $ c $ is then determined by $ S_{6} \circ \phi $, $ \sup_{\bar{M}} \lvert H_{6} \rvert $ as well as the choice of sub- and super-solutions in the proofs of \cite[Thm.~5.1, Cor.~5.1, Cor.~5.2]{XU8}. Note that any diffeomorphism $ \phi $ will not change the supremum of $ \lvert H_{6} \rvert $ on $ \bar{M} $.
\end{proof}
\begin{remark}\label{zero:re1}
The result of Theorem \ref{zero:thm2} indicates that on $ (\bar{M}, g) $ with $ n = \dim \bar{M} \geqslant 3 $, any function that changes sign or identically zero can be realized as a scalar curvature function of some metric $ g $, meanwhile the mean curvature function of $ g $ can be some small enough scaling of any smooth function or zero function, respectively, provided that the manifold admits a metric with zero first eigenvalue of the conformal Laplacian, or equivalently, zero Yamabe invariant \cite[\S1]{ESC}.
\end{remark}
\medskip

% This is the section for prescribed scalar curvature and mean curvature when $ \eta_{1} > 0 $.
\section{Prescribed Scalar and Mean Curvature Functions When $ \eta_{1} > 0 $}
In this section, we seek for a positive, smooth solution of the following PDE
\begin{equation}\label{positive:eqn1}
-a\Delta_{g} u + R_{g} u = Su^{p-1} \; {\rm in} \; M, \frac{\partial u}{\partial \nu} + \frac{2}{p-2} h_{g} u = \frac{2}{p-2} H u^{\frac{p}{2}} \; {\rm on} \; \partial M.
\end{equation}
on compact manifolds $ (\bar{M}, g) $ with non-empty smooth boundary $ \partial M $, $ n = \dim \bar{M} \geqslant 3 $, for given functions $ S, H \in \calC^{\infty}(\bar{M}) $, provided that $ \eta_{1} > 0 $. As we have shown in \cite{XU4}, \cite{XU5} and \cite{XU6}, we need to use local analysis, gluing a super-solution, and then apply monotone iteration scheme here. According to the ``Trichotomy Theorem" in \cite{XU7}, we expect few restrictions on prescribed scalar and mean curvature functions. We will discuss the following case:
\begin{enumerate}[(i).]
\item $ S > 0 $ somewhere in $ M $, and $ H > 0 $ somewhere on $ \partial M $, with $ \eta_{1} > 0 $; 
\item $ S > 0 $ somewhere in $ M $, and $ H \leqslant 0 $ everywhere on $ \partial M $ but $ H \not\equiv 0 $, with $ \eta_{1} > 0 $.
\end{enumerate}
Note that we have discussed the case $ S > 0 $ somewhere and $ H \equiv 0 $ in \cite{XU6}. Currently we do not see how to apply our method to the case mentioned in \cite{ESC3},
\begin{equation}\label{positive:eqn2}
-\Delta_{e} u = 0 \; {\rm in} \; \mathbb{B}^{n}, \frac{\partial u}{\partial \nu} + \frac{2}{p-2} h_{g} u = \frac{2}{p-2} H u^{\frac{p}{2}} \; {\rm on} \; \partial \mathbb{B}^{n}, u > 0
\end{equation}
for some given function $ H $. Escobar showed that there is an obstruction for the choice of $ H $
\begin{equation*}
\int_{\partial \mathbb{B}^{n}} X \cdot \nabla_{g} H dS = 0.
\end{equation*}
Here $ X $ is some conformal Killing field on $ \partial \mathbb{B}^{n} $. With standard Euclidean metric in $ \mathbb{B}^{n} $ and the induced metric on $ \partial \mathbb{B}^{n} $, the first eigenvalue of conformal Laplacian with Robin condition is positive. However, since the right side is zero, we are not able to get a nontrivial local solution of the Dirichlet problem
\begin{equation*}
-\Delta_{e} u = 0 \; {\rm in} \; \Omega, u = 0 \; {\rm on} \; \partial \Omega.
\end{equation*}
Therefore we may need some alternative method to resolve this issue.

However, we can get some interesting results provided that $ S \not\equiv 0 $. According to the detailed analysis in \cite[\S5]{XU6}, we know that there will be obstructions for the choices of prescribed scalar curvature functions on $ \mathbb{S}^{n} \slash \Gamma $ for some Kleinian group $ \Gamma $. The map $ \mathbb{S}^{n} \rightarrow \mathbb{S}^{n} \slash \Gamma $ must be a covering map since otherwise $ \mathbb{S}^{n} \slash \Gamma $ cannnot be a manifold. It follows that $ \mathbb{S}^{n} \slash \Gamma $ has empty boundary, which follows that there will be no obstruction for the choice of prescribed scalar curvature functions on $ (\bar{M}, g) $.

The first result concerns the Case (i) above:
\begin{theorem}\label{positive:thm1}
Let $ (\bar{M}, g) $ be a compact manifold with non-empty smooth boundary $ \partial M $, $ n = \dim \bar{M} \geqslant 3 $. Let $ S_{7} > 0 $ somewhere be any smooth function on $ \bar{M} $. Let $ H_{7} \in \calC^{\infty}(\bar{M}) $ such that $ H_{7} > 0 $ somewhere on $ \partial M $. If $ \eta_{1} > 0 $, then there exists a small enough constant $ c > 0 $ such that (\ref{positive:eqn1}) admits a positive solution $ u \in \calC^{\infty}(\bar{M}) $ with $ S = S_{7} $ and $ H = cH_{7} $. Equivalently, there exists a Yamabe metric $ \tilde{g} = u^{p-2} g $ such that $ R_{\tilde{g}} = S_{7} $ and $ h_{\tilde{g}} = cH_{7} \bigg|_{\partial M} $.
\end{theorem}
\begin{proof} Without loss of generality, we may assume that $ S_{g} > 0 $ and $ h_{g} = h > 0 $ with positive constant $ h $, by Theorem \ref{HL:thm1}. According to Proposition \ref{HL:prop3}, we fix some $ \beta < 0 $ small enough so that $ \eta_{1, \beta} > 0 $ and satisfies
\begin{equation}\label{positive:eqn3}
-a\Delta_{g} \varphi + R_{g} \varphi = \eta_{1, \beta} \varphi \; {\rm in} \; M, \frac{\partial \varphi}{\partial \nu} + \frac{2}{p-2} h_{g} \varphi = 0 \; {\rm on} \; \partial M.
\end{equation} 
Here $ \varphi > 0 $ on $ \bar{M} $. Any scaling of $ \varphi $ solves (\ref{positive:eqn3}). Denote $ \phi = \delta \varphi $ for some $ \delta > 0 $. We choose $ \delta > 0 $ small enough so that
\begin{equation*}
\eta_{1, \beta} \inf_{\bar{M}} \varphi \geqslant \delta^{p-2} \sup_{\bar{M}} S_{7} \sup_{\bar{M}} \varphi^{p-1}.
\end{equation*}
It follows that
\begin{equation*}
-a\Delta_{g} \phi + R_{g} \phi \geqslant S_{7} \phi^{p-1} \; {\rm in} \; M.
\end{equation*}
Fix this $ \delta $. We then choose $ c > 0 $ small enough so that
\begin{equation*}
\beta \phi \geqslant \left( cH_{7} \right) \phi^{\frac{p}{2}} \; {\rm on} \; \partial M.
\end{equation*}
It follows that
\begin{equation}\label{positive:eqn4}
\frac{\partial \phi}{\partial \nu} + \frac{2}{p-2} h_{g} \phi \geqslant \frac{2}{p-2} \left( cH_{7} \right) \phi^{\frac{p}{2}} \; {\rm on} \; \partial M.
\end{equation}
Note that (\ref{positive:eqn4}) still holds for any smaller $ c $. For the sub-solution, we apply Proposition \ref{HL:prop1} or Proposition \ref{HL:prop2}, depending on the vanishing of the Weyl tensor in the interior $ M $, to construct local solution $ u_{0} $ of the Yamabe equation with Dirichlet boundary condition on some domain $ \Omega $. Apply Lemma 3.2 in \cite{XU6}, we can construct a local super-solution $ f $ of the Yamabe equation in $ \Omega $ such that $ f = \phi $ near $ \partial \Omega $. We then define
\begin{equation*}
u_{-} = \begin{cases}  & u_{0} \; {\rm in} \; \Omega \\ & 0 \; {\rm in} \; M \backslash \Omega \end{cases}
\end{equation*}
\begin{equation*}
u_{+} : = \begin{cases} & f \; {\rm in} \; \Omega \\ & \phi \; {\rm in} \; M \backslash \Omega \end{cases}.
\end{equation*}
Since $ u_{-} \equiv 0 $ on $ \partial M $, it follows from the same argument in Lemma 3.1 in \cite{XU6} that $ u_{-} $ is a sub-solution of the (\ref{positive:eqn1}) with $ S = S_{7} $ and $ H = cH_{7} $ for any constant $ c $. According to the construction in Lemma 3.2 of \cite{XU6}, we conclude that $ 0 \leqslant u_{-} \leqslant u_{+} $, $ u_{-} \not\equiv 0 $. In addition, $ u_{-} \in H^{1}(M, g) \cap \calC^{0}(\bar{M}) $, and $ u_{+} \in \calC^{\infty}(\bar{M}) $.
According to (\ref{positive:eqn4}), we have seen that $ u_{+} $ is a super-solution of the (\ref{positive:eqn1}) with $ S = S_{7} $ and $ H = cH_{7} $ for small enough $ c $. Shrinking $ c $, if necessary, so that the hypotheses of smallness of $ c\sup_{\bar{M}} \lvert H_{7} \rvert $ holds. A direct application of Theorem \ref{HL:thm2} indicates the existence of a positive solution $ u \in \calC^{\infty}(\bar{M}) $ with $ S = S_{7} $ and $ H = cH_{7} $.
\end{proof}
\medskip

The proof of the Case (ii) is very similar as in Theorem \ref{positive:thm1}. 
\begin{theorem}\label{positive:thm2}
Let $ (\bar{M}, g) $ be a compact manifold with non-empty smooth boundary $ \partial M $, $ n = \dim \bar{M} \geqslant 3 $. Let $ S_{8} > 0 $ somewhere be any smooth function on $ \bar{M} $. Let $ H_{8} \in \calC^{\infty}(\bar{M}) $ such that $ H_{8} \leqslant 0 $ everywhere on $ \partial M $. If $ \eta_{1} > 0 $, then there exists a small enough constant $ c > 0 $ such that (\ref{positive:eqn1}) admits a positive solution $ u \in \calC^{\infty}(\bar{M}) $ with $ S = S_{8} $ and $ H = cH_{8} $. Equivalently, there exists a Yamabe metric $ \tilde{g} = u^{p-2} g $ such that $ R_{\tilde{g}} = S_{8} $ and $ h_{\tilde{g}} = cH_{8} \bigg|_{\partial M} $.
\end{theorem}
\begin{proof}
Everything is exactly the same as in Theorem \ref{positive:thm1}, except at (\ref{positive:eqn4}), there is no restriction for the choice of the constant $ c $. However, $ c $ should be small enough so that the hypotheses in Theorem \ref{HL:thm2} holds.
\end{proof}
\begin{remark}\label{positive:re1}
The result of Theorem \ref{positive:thm1} and Theorem \ref{positive:thm2} indicate that on $ (\bar{M}, g) $ with $ n = \dim \bar{M} \geqslant 3 $, any function that is positive somewhere can be realized as a scalar curvature function of some metric $ g $, meanwhile the mean curvature function of $ g $ can be some small enough scaling of any smooth function, provided that the manifold admits a metric with positive first eigenvalue of the conformal Laplacian, or equivalently, positive Yamabe invariant \cite[\S1]{ESC}.
\end{remark}

\bibliographystyle{plain}
\bibliography{YamabessSM}

\end{document}